%% file: main.tex
\documentclass[a4paper]{amsart}

\usepackage[labelfont=bf, textfont=normal,
			justification=justified,
			singlelinecheck=false]{caption} 
\usepackage{graphicx} 
\usepackage{hyperref} 
\usepackage{bbm} 
\usepackage[top=2.9cm, bottom=2.9cm,
			inner=3cm, outer=3cm]{geometry}			
\usepackage{parskip} 

\usepackage{mathtools}
\mathtoolsset{showonlyrefs}

\usepackage{enumerate}

\usepackage{booktabs}

\usepackage{enumitem}

\usepackage{pgfplots}
\pgfplotsset{compat=1.16}

\usepackage{subcaption}
\usepackage{pgf,tikz}
\usetikzlibrary{spy, external}
\usepackage{caption}
\usepackage{todonotes}
\usepackage{comment}
\usetikzlibrary{arrows.meta, positioning, shadows, calc}

\usepackage{float}

\newcommand{\triple}{{\vert\kern-0.25ex\vert\kern-0.25ex\vert}}

\usepackage{comment}
\newcommand*\dd{\mathop{}\!\mathrm{d}}
\newcommand{\R}{\mathbb{R}}

\newcommand{\PP}{\mathbb{P}}

\newcommand*{\E}{\mathbb{E}}

\newcommand{\N}{\mathcal{N}}

\newtheorem{theorem}{Theorem}[section]
\newtheorem{corollary}{Corollary}[section]
\newtheorem{proposition}{Proposition}[section]

\newtheorem{remark}{Remark}[section]

\theoremstyle{definition}

\newcommand{\defeq}{\mathrel{:\mkern-0.25mu=}}
\newcommand{\eqdef}{\mathrel{=\mkern-0.25mu:}}

\DeclareMathOperator\spann{span}

\definecolor{crimson2143940}{RGB}{214,39,40}
\definecolor{darkgray176}{RGB}{176,176,176}
\definecolor{darkorange25512714}{RGB}{255,127,14}
\definecolor{forestgreen4416044}{RGB}{44,160,44}
\definecolor{lightgray204}{RGB}{204,204,204}
\definecolor{mediumpurple148103189}{RGB}{148,103,189}
\definecolor{sienna1408675}{RGB}{140,86,75}
\definecolor{steelblue31119180}{RGB}{31,119,180}
\definecolor{orchid227119194}{RGB}{227,119,194}

\usepackage{amssymb}
\usepackage{xcolor}

\usepackage{color}

\usepackage{mathtools}  

\title[Nonlinear filtering based on density approximation and deep BSDE prediction]{
Nonlinear filtering based on density approximation and deep BSDE prediction}

\begin{document}
\author[K.~B{\aa}gmark]{Kasper B{\aa}gmark}
\address{Kasper B{\aa}gmark\\
Department of Mathematical Sciences\\
Chalmers University of Technology and University of Gothenburg\\
SE--412 96 Gothenburg\\
Sweden}
\email{bagmark@chalmers.se}

\author[A.~Andersson]{Adam Andersson} 
\address{Adam Andersson\\ 
Department of Mathematical Sciences\\
Chalmers University of Technology and University of Gothenburg\\
S--412 96 Gothenburg, Sweden\\
and Saab AB \\
S--412 76 Gothenburg, Sweden} 
\email{adam.andersson@chalmers.se} 

\author[S.~Larsson]{Stig Larsson}
\address{Stig Larsson\\
Department of Mathematical Sciences\\
Chalmers University of Technology and University of Gothenburg\\
SE--412 96 Gothenburg\\
Sweden}
\email{stig@chalmers.se}


\begin{abstract}
    A novel approximate Bayesian filter based on backward stochastic differential equations is introduced. It uses a nonlinear Feynman--Kac representation of the filtering problem and the approximation of an unnormalized filtering density using the well-known deep BSDE method and neural networks. The method is trained offline, which means that it can be applied online with new observations. A hybrid \emph{a priori}-\emph{a posteriori} error bound is proved under a parabolic Hörmander condition. The theoretical convergence rate is confirmed in two numerical examples.
\end{abstract}
\maketitle
\section{Introduction}
We consider a hybrid system, which couples a state equation in continuous time with measurements in discrete time:
\begin{align}
    \label{introduction eq: state}
    S_t
    &=
    S_0
    +
    \int_0^t
    \mu
    (S_s)
    \dd s
    +
    \int_0^t
    \sigma
    (S_s)
    \dd B_s
    , \quad
    t\in [0,T],
    \\
    \label{introduction eq: obs}
    O_k
    &=
    h(S_{t_k})
    +
    V_k
    ,\quad
    k=1,\dots,K
    .
\end{align}
Here, the process $S$ evolves according to a Stochastic Differential Equation (SDE) with time-independent coefficients $\mu$ and $\sigma$, driven by Brownian motion $B$, and with initial distribution $\pi_0$. Observations are made at discrete times $t_1,\dots,t_K$, with a measurement function $h$ and additive independent noise $V_1,\dots,V_K$.
The filtering problem seeks to compute for every $k\in\{1,\dots,K\}$ the conditional distribution of the latent state $S_{t_k}$ given observations $O_{1:k}$. In other words, the goal is to estimate the conditional density $p(S_{t_k} \mid O_{1:k})$, which encodes all information available about the system at time $t_k$. This problem is central to Bayesian inference for dynamical systems and is commonly referred to as Bayesian filtering.

Classical, exact or approximate, solutions to the filtering problem---such as the Kalman filter and its extensions---are effective in low dimensions or when the true solution is close to Gaussian. However, in the nonlinear, non-Gaussian regime, such methods are either biased (e.g., the extended Kalman filter) or suffer from the curse of dimensionality (e.g., particle filters).

This paper introduces a novel approximation scheme for the Bayesian filtering problem with the aim of mitigating the curse of dimensionality. We formulate the filtering problem recursively as a sequence of prediction steps governed by the Fokker--Planck Partial Differential Equation (PDE), intertwined with update steps via Bayes' formula. By leveraging the connection between parabolic PDEs and Backward Stochastic Differential Equations (BSDE), we obtain a probabilistic representation of the prediction step.  The resulting BSDEs are solved numerically by using the deep BSDE method \cite{E_2017}, while the update step is implemented in closed form at each observation time $t_k$. This leads to a recursive approximation scheme.

Our main contributions are as follows:
\begin{enumerate}[label=\arabic*.]
    \item We propose a density-based recursive filter, which is based on sequentially employing the deep BSDE method with observations as input.
    \item We prove a hybrid \emph{a priori}-\emph{a posteriori} error estimates for the method, under a parabolic Hörmander condition, extending the convergence results of~\cite{han_convergence} to a setting involving a sequence of uncoupled forward-backward systems.
    \item We demonstrate the theoretical convergence order in numerical experiments.
\end{enumerate}
The numerical experiments focus on convergence behaviour and are therefore restricted to one-dimensional examples. A complementary numerical study assessing scalability and comparison with classical filtering methods was conducted in~\cite{baagmark2025high}. There, the method was shown to scale successfully, including a nonlinear example in dimension $100$. The code for the implementation is publicly available\footnote{\url{https://github.com/bagmark/deep-density-filtering}}.

\subsection{Bayesian filtering}
Bayesian filtering has broad applicability across science and engineering. 
It plays a foundational role in target tracking~\cite{barshalom2001estimation, blackman1999design, cui2005comparison, goodman1997mathematics},
robot localization~\cite{thrun2005probabilistic},
and neural decoding~\cite{kording2004bayesian}. In finance, it helps to estimate latent variables like volatility~\cite{johannes2009mcmc},
and in geosciences, it enables data assimilation through ensemble filters~\cite{evensen2009data}.
Computer vision tasks like object tracking can be tackled via particle filters~\cite{isard1998condensation}. For multi-sensor navigation and fusion of inertial data, Bayesian filters are a staple~\cite{maybeck1979stochastic},
and in epidemiology, they support real-time estimation of latent disease spread~\cite{jasra2005population}.

\paragraph*{Bayesian approaches}
When the system dynamics and the observations \eqref{introduction eq: state}--\eqref{introduction eq: obs} are linear and Gaussian, the conditional distribution remains Gaussian and is computed recursively via the Kalman filter \cite{Kalman_Bucy}. In the nonlinear case, there are several approximation techniques. The extended Kalman filter linearizes the dynamics through Taylor expansion, while the unscented Kalman filter employs sigma-point sampling. These filters approximate the posterior as Gaussian and often perform well in moderate dimensions when the nonlinearity is weak. Beyond Gaussian approximations, the ensemble Kalman filter uses particle ensembles to approximate moments of the posterior, though it still relies on linear-Gaussian observation models. Particle filters (also known as sequential Monte Carlo methods) represent the gold standard in nonlinear filtering, by propagating samples (particles) and updating them via importance weights. Sampling-based methods based on both particle filters and other techniques have met great interest and success in the last few years \cite{chopin2023computational, corenflos2024particlemala, corenflos2024conditioning, finke2023conditional, 
naesseth2019high, schauer2017guided, zhao2024conditional}. However, particle filters for state estimation often scale poorly with state dimension due to sample degeneracy and exponential growth in variance, which are hallmarks of the curse of dimensionality \cite{snyder2011particle,snyder2015performance}.

\paragraph*{PDE approaches} 
Classical filtering methods consist of a prediction step and an update step, where the prediction can be obtained by solving the Fokker--Planck equation. This is explored in \cite{challa2000nonlinear, demissie2016nonlinear} among others. However, solving high-dimensional PDEs remains a fundamental challenge in computational mathematics. Classical grid-based methods quickly become infeasible due to the curse of dimensionality. In recent years, several deep learning-based approaches have emerged to tackle this limitation by leveraging probabilistic formulations and neural networks. Among these, physics-informed neural networks \cite{raissi2019physics}, the deep Ritz method \cite{e2017deep}, neural operators \cite{Lu_2021}, and deep Picard iterations \cite{han2024deep} offer alternatives for PDE approximation that are less mesh-dependent. However, their performance often degrades in high dimensions, especially when global approximations over the entire domain are required \cite{krishnapriyan2021characterizing, luo2025physics}. The challenge intensifies in probabilistic formulations, where the PDE solution corresponds to a probability density. For example, in high dimensions, a Gaussian density becomes sharply concentrated, with values diminishing exponentially in $d$, resulting in severe numerical underflow and instability.

Two approaches that explicitly target high-dimensional problems, such as the ones we study, are the deep splitting method \cite{Arnulf_PDE} and the deep BSDE method \cite{E_2017}. Both exploit probabilistic representations of PDEs through Feynman--Kac formulas and use neural networks to approximate solutions. These methods have demonstrated effectiveness in high-dimensional settings, and recent extensions of the deep BSDE method have further improved its robustness \cite{andersson2025deep, kapllani2025backward}. Deep splitting solves a sequence of optimization problems along a time grid, but suffers from error accumulation and challenging training dynamics. In \cite{bagmark_1, Arnulf, Crisan_Lobbe}, deep splitting methods were applied to the Zakai equation, whose solution is a filtering density in the case of continuous in time observations. Later, in \cite{baagmark2024convergent}, the method was extended to the case of discrete observations, identical to the problem formulation of this paper. In the current paper, we introduce an approximate filter that uses the deep BSDE method for the prediction step.

\subsection{Outline of the paper}
Section~\ref{section: background} introduces the problem setting, notation, and the derivation of the new filtering method. In Section~\ref{section: error analysis}, we present the convergence theorem along with its proof. Section~\ref{section: numerical results} contains our numerical experiments with empirical convergence results and a discussion of the practical approximations required for the implementation as neural networks. In Section~\ref{section: conclusion} we summarize the results and give a brief outlook on possible extensions. Implementation details are provided in the Appendix.

\section{The BSDE formulation of the filtering problem}
\label{section: background}
\subsection{Notation}
We use a fixed time horizon $T > 0$ and integers $d, d', K \geq 1$ throughout this paper.
We denote by $\langle x, y \rangle$ and $\|z\|$ the standard Euclidean inner product and norm when $x, y, z \in \mathbb{R}^d$ and the Frobenius norm when $z \in \mathbb{R}^{d \times d}$. The function space $C^{k,n}([0,T] \times \mathbb{R}^d; \mathbb{R})$ consists of real-valued functions defined on $[0,T] \times \mathbb{R}^d$ that are $k$ times continuously differentiable in time and $n$ times continuously differentiable in space, without requiring the existence of mixed derivatives with respect to $t$ and $x$. The subspace $C_{\mathrm{b}}^{k,n}([0,T] \times \mathbb{R}^d; \mathbb{R})$ contains all, not necessarily bounded, functions with bounded derivatives. The notation is adapted in obvious ways, when $k$ or $n$ equals $\infty$, or when the functions are independent of time. Given a measure space $(A, \mathcal{B}, \nu)$ and a Banach space $U$, for $p \in [1, \infty]$, we use the standard notation for the Bochner space $L^p(A; U)$. For vector fields $V, W \in C^1(\mathbb{R}^d; \mathbb{R}^d)$, we define their Lie bracket as
\begin{align*}
    \big[V,W\big](x)
    =
    \mathrm{D}W(x)V(x)
    -
    \mathrm{D}V(x)W(x)
    ,\quad
    x\in\R^d
    ,
\end{align*}
where $\mathrm{D}V(x)$ denotes the Jacobian matrix of $V$ at $x$.
We say that a collection of vector fields $V_0, \dots, V_n \in C_{\mathrm{b}}^\infty(\mathbb{R}^d; \mathbb{R}^d)$ satisfies the parabolic Hörmander condition if, for every $x \in \mathbb{R}^d$,
\begin{align}
\begin{split} \label{eq: parabolic hörmander}
    &\spann \Big\{
    V_{j_0}(x),
    \big[V_{j_1}(x)
    ,V_{j_2}(x)\big],
    \Big[\big[V_{j_1}(x)
    ,V_{j_2}(x)\big]
    ,V_{j_3}(x)\Big],
    \dots ;
    \\
    &\hspace{8em}
    j_0\in \{1,\dots,n\},\,
    j_i \in \{0,\dots,n\}\text{ for all } 
    i=1,\dots,n
    \Big\}
    =
    \R^d.
\end{split}
\end{align}
Given a function $O\colon \mathbb{N}\to\mathbb{R}^{d'}$, we define the $d'\times k$ matrix $O_{1:k}=[O_{1},\dots,O_{k}]$, if $k\geq1$, and $O_{1:0}=\emptyset$, otherwise. To simplify the notation, we introduce the observation spaces $\mathbb{O}_k \coloneqq \R^{d' \times k}$ and $\mathbb{O} \coloneqq \bigcup_{k=1}^K \mathbb{O}_k$.

\subsection{Relating the Fokker--Planck equation to BSDEs}
Let the quadruple $(\Omega,\mathcal{A},(\mathcal{F}_t)_{0\leq t\leq T},\mathbb{P})$ be a complete filtered probability space and let the filtration $\mathcal{F} := (\mathcal{F}_t)_{0\leq t\leq T}$ be with respect to two independent $d$-dimensional Brownian motions $B$ and $W$. We assume that the initial state $S_0$ of \eqref{introduction eq: state} is distributed according to some probability density $\pi_0$. The observation equation \eqref{introduction eq: obs} implicitly leads to the likelihood function $L(o,x) \defeq p(O_k = o \mid S_{t_k}=x)$.

In the analysis and derivations of this paper, we require the following two conditions.
\begin{enumerate}[label=(\roman*)]
    \item \textbf{Smooth boundedness.}
    \label{eq: condition smooth}
    The coefficients and functions in  
    \eqref{introduction eq: state} and \eqref{introduction eq: obs} are smooth, with all derivatives of order at least one bounded. More precisely, it holds
    \begin{align*}
        \mu 
        &\in 
        C_{\mathrm{b}}^\infty(
        \mathbb{R}^d; 
        \mathbb{R}^d),
        \hspace{0.30em}
        \sigma 
        \in 
        C_{\mathrm{b}}^\infty(
        \mathbb{R}^d; 
        \mathbb{R}^{d \times d}),
        \hspace{0.30em}
        \pi_0 
        \in 
        C_{\mathrm{b}}^\infty(
        \mathbb{R}^d; 
        \mathbb{R}),
        \hspace{0.30em}
        L 
        \in 
        C_{\mathrm{b}}^{0,\infty}(
        \mathbb{R}^{d'} \times \mathbb{R}^d; 
        \mathbb{R}).
    \end{align*}
    Moreover, $L$ is uniformly bounded, i.e., there exists a constant $C>0$ such that
    \begin{align} 
    \label{eq: uniform bound L}
        \sup_{o\in \R^{d'}}
        \big\|
        L(o)
        \big\|_{L^\infty(\R^d;\R)}
        \leq
        C
        .
    \end{align}
    \item \textbf{Hypoellipticity.} \label{eq: condition hörmander}
    The coefficients $\mu$ and $\sigma$ satisfy the Hörmander condition. More precisely, by writing $\sigma = [\sigma_1, \dots, \sigma_d]$ for the column-wise decomposition and defining the vector fields
    \begin{align*}
        V_i 
        &\defeq 
        \sigma_i,
        \quad 
        i = 1, \dots, d,
        \qquad
        V_0 
        \defeq 
        \mu 
        + 
        \frac{1}{2} 
        \sum_{j = 1}^d 
        \mathrm{D} 
        V_j  \, V_j,
    \end{align*}
    we assume that the family $\{V_0, \dots, V_d\}$ satisfies the parabolic Hörmander condition \eqref{eq: parabolic hörmander} for all $x \in \mathbb{R}^d$.
\end{enumerate}
The state equation \eqref{introduction eq: state} has a unique solution $S = (S_t, ~t\in[0,T])$ under condition \eqref{eq: condition smooth} \cite{Oksendal}. We fix the discrete time points, $0=t_0 < t_1 < \dots < t_K = T$, uniformly, i.e., $t_{k} - t_{k-1} = \frac{T}{K}$ for all $k=1,\dots,K$.
The continuous-discrete filtering problem that we consider consists of computing the conditional probability density of $S$ at the time points $t_k$, $k=1,\dots,K$, given observations $o_{1:k}\in \mathbb{O}_k$. In order to introduce the equation for the filtering density, we begin by recalling the generator $A$ associated to $S$, and its adjoint operator $A^*$, which are defined for $\varphi\in C^2(\R^d;\R)$, by  
\begin{align*}
    A\varphi 
    = 
    \frac{1}{2}\sum_{i,j=1}^d a_{ij}\,
    \frac{\partial^2 \varphi}{\partial x_i \partial x_j} 
    + \sum_{i=1}^d \mu_i \, 
    \frac{\partial \varphi}{\partial x_i}
    ,\quad
    A^* \varphi 
    = 
    \frac{1}{2}\sum_{i,j=1}^d 
    \frac{\partial^2}{\partial x_i \partial x_j} 
    (a_{ij}\varphi) 
    - \sum_{i=1}^d 
    \frac{\partial}{\partial x_i} 
    (\mu_i \varphi),
\end{align*}
where $a=\sigma\sigma^\top\in\mathbb{R}^{d\times d}$. 
We now state the filtering equations, whose solution is the desired conditional density over time $[0,T]$, yielding the filtering densities at the time points $(t_k)_{k=1}^K$. These equations define functions $p_k\in C([t_k,t_{k+1}]\times\R^d,\R)$ sequentially for $k=0,\dots,K-1$, and are parameterized by $\mathbb{O}$. We begin, for $x\in\R^d$, by initializing
\begin{align} \label{eq: global Fokker--Planck with update}
    p_{0}
    (0,x)
    =
    \pi_0(x)
\end{align}
and continue by solving, for $k=0,\dots,K-1$ and $o_{1:k}\in\mathbb{O}_k$,
\begin{align}
    \label{eq: Fokker--Planck}
    p_{k}(t,x,o_{1:k}) 
    &= 
    p_{k}(t_{k},x,o_{1:k})
    +
    \int_{t_{k}}^{t} 
    A^* p_{k}(s,x,o_{1:k}) 
    \dd s
    ,\quad
    t\in (t_{k},t_{k+1}]
    ,
    \\
    \label{eq: FP-update}
    p_{k+1}
    (t_{k+1},x,o_{1:k+1})
    &
    =
    p_{k}
    (t_{k+1},x,o_{1:k})
    L(o_{k+1},x)
    .
\end{align}
The first equation, \eqref{eq: Fokker--Planck}, is the prediction equation, which is a deterministic PDE known as the Fokker--Planck equation. This is followed by the update equation, \eqref{eq: FP-update}, which is immediately tractable up to a normalizing constant $C = C(o_{1:k+1}) = \big(\int_{\R^d} p_{k}(t_{k+1},z,o_{1:k})L(o_{k+1},z) \dd z \big)^{-1}$ that is omitted here. Therefore, by \eqref{eq: global Fokker--Planck with update}--\eqref{eq: FP-update}, we obtain the (unnormalized) filtering density. In practical settings, it is sometimes necessary or useful to normalize the density, but it is not necessary for the theoretical results of this paper and is therefore omitted until Section~\ref{section: numerical results}. 

Next, we quote a regularity result for $p$ from \cite[Proposition~2.1]{baagmark2024convergent}. Throughout the remainder of this section, we assume that conditions~\eqref{eq: condition smooth} and~\eqref{eq: condition hörmander} hold.
\begin{proposition}
    \label{prop: properties of p} 
    There exists a unique $p_{k}\in L^\infty(\mathbb{O}_k;C([t_k,t_{k+1}]\times\R^d,\R))$, $k=0,\dots,K-1$, satisfying \eqref{eq: global Fokker--Planck with update}--\eqref{eq: FP-update}. Moreover, $p_{k}(o_{1:k})\in C_{\mathrm b}^{1,\infty}([t_k,t_{k+1}]\times \R^d; \R)$ for all $k=0,\dots,K-1$ and $o_{1:k}\in\mathbb{O}_k$. 
\end{proposition}
In particular, we note that the filtering densities $p_k(t_k,\cdot, o_{1:k})$ belong to $C_{\mathrm{b}}^{\infty}(\R^d; \R)$ for all $k = 1, \dots, K$ and $o_{1:k} \in \mathbb{O}_k$. 

Our next result is a nonlinear Feynman--Kac formula for the filtering problem, with a BSDE reformulation of the prediction equation \eqref{eq: Fokker--Planck}. In order to formulate it, we introduce an auxiliary drift $b\in C_{\mathrm{b}}^\infty(\R^d;\R^d)$ and process $X$, satisfying 
\begin{align}
\label{eq: auxiliary process}
    X_t
    =
    X_0
    +
    \int_0^t
    b
    (X_s)
    \dd s
    +
    \int_0^t
    \sigma
    (X_s)
    \dd W_s,
    \quad
    t\in [0,T],\ \text{$\PP$-a.s.}
\end{align}
Here, $W$ is a $d$-dimensional Brownian motion independent of $B$ and $X_0\sim \pi_0$. 
We denote by $\pi_t$ the unconditional distribution of $X$ at time $t\in[0,T]$. 
The corresponding generator $A_b$, is defined, for $\varphi\in C^2(\R^d;\R)$, by
\begin{align*}
    A_b\varphi 
    = 
    \frac{1}{2}\sum_{i,j=1}^d a_{ij}\,
    \frac{\partial^2 \varphi}{\partial x_i \partial x_j} 
    + 
    \sum_{i=1}^d 
    b_i \, 
    \frac{\partial \varphi}{\partial x_i}.
\end{align*}
Defining the function $f_b\colon \R^d\times\R\times \R^d \to \R$ by 
\begin{align*} 
\begin{split}
    f_b
    (x,u,v) 
    &= 
    \sum_{i=1}^d 
    \bigg(
    \sum_{j=1}^d
    \frac
    {\partial a_{ij}(x)}
    {\partial x_j }
    \bigg)
    \, v_i
    + 
    \frac{1}{2}
    \sum_{i,j=1}^d 
    \frac{
    \partial^2 
    a_{ij}(x)
    }
    {
    \partial x_i 
    \partial x_j
    }
    \,  u 
    - 
    \sum_{i=1}^d 
    \frac{
    \partial \mu_i(x)
    }
    {
    \partial x_i
    }
    \, u
    \\
    &\hspace{2em}
    - 
    \sum_{i=1}^d
    \mu_i(x) 
    \, v_i
    -
    \sum_{i=1}^d
    b_i(x) 
    \, v_i
    ,
\end{split}
\end{align*}
we see, for $\varphi\in C^2(\R^d;\R)$ and $x\in\R^d$, that
\begin{align} \label{eq: f relation}
    (A^* \varphi)
    (x) 
    &= 
    (A_b\varphi)(x) 
    + 
    f_b
    (x,
    \varphi(x),
    \nabla 
    \varphi(x)
    ).
\end{align}
We note that the introduction of an auxiliary drift $b$ extends the original deep BSDE scheme from \cite{E_2017}, where $b=\mu$ was the only choice. Our extension shares similarities with the approach from \cite{andersson2025deep}. In Section~\ref{section: error analysis} we utilize this freedom by choosing $b$ as
\begin{align}
\label{eq: mu-bar choice}
    b_i
    &=
    -
    \mu_i
    +
    \sum_{j=1}^d
    \frac{\partial a_{ij}}
    {\partial x_j }
    ,
    \quad
    i=1,\dots,d.
\end{align}
For this choice of $b$, the $v$-dependent terms in $f_b$ cancel. We thus obtain a reduced function $f\colon\R^d\times\R\to\R$ defined by
\begin{align} \label{eq: simplified f} 
\begin{split}
    f
    (x,u) 
    &= 
    \frac{1}{2}
    \sum_{i,j=1}^d 
    \frac{
    \partial^2 
    a_{ij}(x)
    }
    {
    \partial x_i 
    \partial x_j
    }
    \,  u 
    - 
    \sum_{i=1}^d 
    \frac{
    \partial \mu_i(x)
    }
    {
    \partial x_i
    }
    \, u
    .
\end{split}
\end{align}
The derivation of the nonlinear Feynman--Kac formula is similar to the classical one in \cite[Proposition 4.3]{BSDE_in_finance}.

\begin{theorem}[Nonlinear Feynman--Kac formula for the filtering problem]
    \label{theorem: BSDE cont formulation}
    Let $p_k$, $k=0,\dots,K$, be the solution to \eqref{eq: global Fokker--Planck with update}--\eqref{eq: FP-update}. 
    For $k\in\{0,\dots,K-1\}$, and $o_{1:k}\in \mathbb{O}_k$, let $(X^k,Y^k,Z^k)$ be the unique process that satisfies, for $t\in[t_k,t_{k+1}]$, $\mathbb{P}$-a.s. 
    \begin{align}
    \begin{split} \label{theorem eq: BSDE cont formulation}
        X_t^k
        &=
        X_{t_k}^k
        +
        \int_{t_k}^t
        b
        (X_s^k)
        \dd s
        +
        \int_{t_k}^t
        \sigma
        (X_s^k)
        \dd W_s,
        \\
        Y_t^k
        &=
        g_k
        (X_{t_{k+1}}^k
        ,
        o_{1:k})
        +
        \int_t^{t_{k+1}}
        f_b(
        X_s^k,
        Y_s^k,
        Z_s^k
        )
        \dd s
        -
        \int_t^{t_{k+1}}
        (Z_s^k)^\top
        \sigma(X_s^k)
        \dd W_s,
    \end{split}
    \end{align}
    where $X_{t_k}^k\sim \pi_{t_k}$ and the terminal condition $g_k$ is defined by
    \begin{align}
    \label{theorem eq: g_k}
        g_k(x,o_{1:k})
        &=
        \begin{cases}
            p_{k-1}
            (t_k,
            x,
            o_{1:k-1})
            L
            (
            o_k,
            x
            ),
            & 
            k\geq 1,
            \\
            \pi_0(x),
            & 
            k=0.
        \end{cases} 
    \end{align}
    Then we have, for $t\in[t_k,t_{k+1}]$, $\mathbb{P}$-a.s.
    \begin{align} \label{theorem eq: BSDE equivalence}
        &\big(
        p_k
        (t,
        X_{t_{k+1}+t_k-t}^k,
        o_{1:k})
        , 
        \nabla
        p_k
        (t,
        X_{t_{k+1}+t_k-t}^k,
        o_{1:k}) 
        \big)
        =
        (
        Y_{t_{k+1}+t_k - t }^k
        ,
        Z_{t_{k+1}+t_k - t }^k
        )
        .
    \end{align}
\end{theorem}
\begin{proof}
    We fix $k\in\{0,\dots,K\}$, $o_{1:k}\in\mathbb{O}_k$, and omit the $o$-notation whenever suitable. We introduce the reparameterized solution $q_k(t) = p_k(t_{k+1}+t_k-t)$, for $t\in [t_k,t_{k+1}]$. By \eqref{eq: f relation} and \eqref{eq: Fokker--Planck}--\eqref{eq: FP-update} it follows that $q_k$, for $t\in[t_k,t_{k+1}]$ and $x\in\R^d$, satisfies
    \begin{align*}
        \frac{\partial}
        {\partial t}
        q_k
        (t,x)
        +
        A_b
        q_k
        (t,x)
        &=
        -
        f_b
        \big(
        x,
        q_k
        (t,x),
        \nabla 
        q_k
        (t,x)
        \big)
        ,
        \\
        q_k(t_{k+1},x,o_{1:k})
        &=
        g_k
        (x,o_{1:k})
        .
    \end{align*}
    By Proposition~\ref{prop: properties of p}, the solution $p_k$ belongs to $C_{\mathrm{b}}^{1,2}([t_k,t_{k+1}]\times\R^d;\R)$ and thus so does $q_k$. We use Itô's formula for $q_k(t,X_t)$, $t\in[t_k,t_{k+1}]$, and $s\in[t_k,t]$, to obtain $\mathbb{P}$-a.s.\
    \begin{align*}
        q_k
        (t,
        X_t)
        &=
        q_k
        (s,
        X_{s})
        +
        \int_{s}^t
        \Big(
        \frac{\partial}{\partial u}
        q_k
        (u,
        X_u)
        +
        A_b
        q_k
        (u,
        X_u)
        \Big)
        \dd u
        +
        \int_{s}^t
        \nabla
        q_k
        (u,
        X_u)^\top 
        \sigma(
        X_u)
        \dd W_u
        \\
        &=
        q_k
        (s,
        X_{s})
        -
        \int_{s}^t
        f_b
        \big(
        X_u,
        q_k
        (u,
        X_u),
        \nabla 
        q_k
        (u,
        X_u)
        \big)
        \dd u
        +
        \int_{s}^t
        \nabla
        q_k
        (u,
        X_u)^\top 
        \sigma(
        X_u)
        \dd W_u.
    \end{align*}
    Defining $Y_t = q_k(t,X_t)$ and $Z_t = \nabla q_k(t,X_t) $ we get
    \begin{align*}
        Y_t
        &=
        Y_{s}
        -
        \int_{s}^t
        f_b(
        X_u,
        Y_u,
        Z_u
        )
        \dd u
        +
        \int_{s}^t
        Z_u^\top
        \sigma(X_u)
        \dd W_u.
    \end{align*}
    In particular, with $t=t_{k+1}$ we have 
    \begin{align*}
        Y_{t_{k+1}}
        =
        Y_s
        -
        \int_s^{t_{k+1}}
        f_b(
        X_u,
        Y_u,
        Z_u
        )
        \dd u
        +
        \int_s^{t_{k+1}}
        Z_u^\top
        \sigma(X_u)
        \dd W_u
        ,\quad
        s\in[t_k,t_{k+1}].
    \end{align*}
    Since $Y$ satisfies the terminal condition $Y_{t_{k+1}} = g_{k}(X_{t_{k+1}},o_{1:k})$, we can write the equations for $(X,Y,Z)$ as an uncoupled Forward Backward Stochastic Differential Equation (FBSDE) in the interval $[t_k,t_{k+1}]$,
    \begin{align*} 
        X_t
        &=
        X_{t_k}
        +
        \int_{t_k}^t
        b
        (X_s)
        \dd s
        +
        \int_{t_k}^t
        \sigma
        (X_s)
        \dd W_s,
        \\
        Y_t
        &=
        g_{k}
        (X_{t_{k+1}}
        ,
        o_{1:k})
        +
        \int_t^{t_{k+1}}
        f_b(
        X_s,
        Y_s,
        Z_s
        )
        \dd s
        -
        \int_t^{t_{k+1}}
        Z_s^\top
        \sigma(X_s)
        \dd W_s.
    \end{align*}
    Under conditions \eqref{eq: condition smooth} and \eqref{eq: condition hörmander}, this FBSDE has a unique solution, which we denote $(X^k,Y^k,Z^k)$ and  \eqref{theorem eq: BSDE equivalence} is satisfied by construction, see \cite[Proposition 4.3]{BSDE_in_finance}.
\end{proof}
We note that, while $X = (X^k)_{k=0}^{K-1}$ is defined in the entire interval $[0,T]$ as a continuous solution, the processes $(Y^k,Z^k)_{k=0}^{K-1}$ are different processes defined in separate intervals. In the following, we denote by $(X^{k,x},Y^{k,x},Z^{k,x})$ the process $(X^k,Y^k,Z^k)$ conditioned on $X_{t_k}^k = x$, thus introducing a dependence on $x$ also in $(Y^k,Z^k)$.  By the stochastic representation of $p$ in Theorem~\ref{theorem: BSDE cont formulation} the following deterministic representation is easily obtained.
\begin{corollary}
\label{corollary: feynman--kac formula}
    Assume the setting of Theorem~\ref{theorem: BSDE cont formulation}. For $k=0,\dots,K-1$, $o_{1:k}\in\mathbb{O}_k$, it holds
    \begin{align*}
        p_k
        (t_{k+1},
        x,
        o_{1:k}
        )
        &=
        \E
        \Bigg[
        g_k(
        X_{t_{k+1}}^{k,x},
        o_{1:k}
        )
        +
        \int_{t_k}^{t_{k+1}}
        f_b
        (
        X_s^{k,x},
        Y_s^{k,x},
        Z_s^{k,x}
        )
        \dd s
        \Bigg]
        ,\quad
        x\in\R^d
        .
    \end{align*}
\end{corollary}
\begin{proof}
    We begin by letting $k\in\{0,\dots,K-1\}$, $o_{1:k}\in\mathbb{O}_k$, and $t={t_{k+1}}$. By taking the conditional expectation with respect to $X_{t_k}^k = x$ in the first component of \eqref{theorem eq: BSDE equivalence}, we obtain
    \begin{align*}
        p_k
        (t_{k+1},
        x,
        o_{1:k}
        )
        &=
        \E
        \big[
        Y_{t_k}^{k}
        \mid
        X_{t_k}^k
        =
        x
        \big]
        \\
        &=
        \E
        \Bigg[
        g_k(
        X_{t_{k+1}}^{k},
        o_{1:k}
        )
        +
        \int_{t_k}^{t_{k+1}}
        f_b
        (
        X_s^{k},
        Y_s^{k},
        Z_s^{k}
        )
        \dd s
        \,\Bigg\vert 
        X_{t_k}^k
        =
        x
        \Bigg]
        \\
        &=
        \E
        \Bigg[
        g_k(
        X_{t_{k+1}}^{k,x},
        o_{1:k}
        )
        +
        \int_{t_k}^{t_{k+1}}
        f_b
        (
        X_s^{k,x},
        Y_s^{k,x},
        Z_s^{k,x}
        )
        \dd s
        \Bigg].
    \end{align*}
    This completes the proof.
\end{proof}
\subsection{Optimization formulation}
The BSDE formulation \eqref{theorem eq: BSDE cont formulation} introduces a new possibility for approximating the prediction step probabilistically. To this end, we reformulate \eqref{theorem eq: BSDE cont formulation} as an equivalent optimization problem. This  means that we extend the minimization problem of the original deep BSDE formulation in \cite{E_2017} to also include observations. 
\begin{proposition}
    \label{prop: continuous BSDE minimization}
    Assume the setting of Theorem~\ref{theorem: BSDE cont formulation} and let $O\colon\Omega \to \mathbb{O}$ be random observations, independent of $X$, and whose distribution is determined by \eqref{introduction eq: obs}.
    The solution $p_k$ to \eqref{eq: global Fokker--Planck with update}--\eqref{eq: FP-update}, is then given by $p_k(t) = u_k^*(t_{k+1}+t_k-t)$, for  $k=0,\dots,K-1$, $t\in[t_k,t_{k+1}]$, where $u_k^*$ is the solution to the minimization problem
    \begin{align*} 
        &\hspace{-4em}
        \underset{u\in 
        L^\infty(\mathbb{O}_k;
        C([t_k,t_{k+1}]
        \times\R^d
        ;\R))}{\mathrm{min}}\  
        \E
        \Bigg[
        \Big|
        u(t_{k+1},
        X_{t_{k+1}}^k,
        O_{1:k}
        )
        -
        g_k
        (X_{t_{k+1}}^k,
        O_{1:k})
        \Big|^2
        \Bigg]
        \\
        X_t^k
        &=
        X_{t_k}^k
        +
        \int_{t_k}^t
        b
        (X_s^k)
        \dd s
        +
        \int_{t_k}^t
        \sigma
        (X_s^k)
        \dd W_s
        ,\quad
        t\in[t_k,t_{k+1}],
        \\
        u(t,X_t^k, O_{1:k})
        &=
        u(t_k,X_{t_k}^k, O_{1:k})
        -
        \int_{t_k}^{t}
        f_b
        \big(
        X_s^k,
        u(s,X_s^k, O_{1:k}),
        \nabla 
        u(s,X_s^k, O_{1:k})
        \big)
        \dd s
        \\
        &\hspace{8em}+
        \int_{t_k}^{t}
        \nabla u(s,X_s^k, O_{1:k})^{\top}
        \sigma
        (X_s^k)
        \dd W_s
        ,\quad
        t\in[t_k,t_{k+1}].
    \end{align*}
\end{proposition}
\begin{proof}
    Follows directly by substitution and Theorem~\ref{theorem: BSDE cont formulation}.
\end{proof}
Next, we formulate an equivalent optimization problem by splitting $u$ and $\nabla u$ into two separate functions $w$ and $v$, where $w= u(t_{k})$ and $v = \nabla u$ over $[t_k,t_{k+1}]$. As we shall see, this makes the discretization schemes introduced below feasible for a more practical implementation that avoids multiple layers of automatic differentiation. The new optimization problem reads as follows.
\begin{align} \label{prop eq: continuous BSDE minimization split}
\begin{split} 
    &\hspace{-4em}
    \underset{
    \substack{
    w\in 
    L^\infty(\mathbb{O}_k;
    C(
    \R^d
    ;\R))
    \\
    v\in 
    L^\infty(\mathbb{O}_k;
    C([t_k,t_{k+1}]
    \times\R^d
    ;\R^d))
    }
    }
    {\mathrm{min}}\  
    \E
    \Bigg[
    \Big|
    Y_{t_{k+1}}^{k}
    -
    g_k
    (X_{t_{k+1}}^k,
    O_{1:k})
    \Big|^2
    \Bigg]
    \\
    X_t^k
    &=
    X_{t_k}^k
    +
    \int_{t_k}^t
    b
    (X_s^k)
    \dd s
    +
    \int_{t_k}^t
    \sigma
    (X_s^k)
    \dd W_s
    ,\quad
    t\in[t_k,t_{k+1}],
    \\
    Y_{t}^{k}
    &=
    w(X_{t_k}^k, O_{1:k})
    -
    \int_{t_k}^{t}
    f_b
    \big(
    X_s^k,
    Y_{s}^{k},
    v(s,
    X_s^k, 
    O_{1:k})
    \big)
    \dd s
    \\
    &\hspace{8em}+
    \int_{t_k}^{t}
    v(s,
    X_s^k, 
    O_{1:k})^\top
    \sigma
    (X_s^k)
    \dd W_s
    ,\quad
    t\in[t_k,t_{k+1}].
\end{split}
\end{align}
It is easy to see that, for $k=\{0,\dots,K-1\}$, $(u_k^*(t_{k}),\nabla u_k^*)$, where $u_k^*$ is the solution from Proposition~\ref{prop: continuous BSDE minimization}, solves \eqref{prop eq: continuous BSDE minimization split}. 

At this stage, we have obtained an optimization problem over both time, space, and observation space. Next, we introduce the Euler--Maruyama approximations $\mathcal{X}$ and $\mathcal{Y}$ of $X$ and $Y$, respectively. Define a finer time partition $\{ t_{k,n} \}_{n=0, k=0}^{N,K-1} $ of the interval $[0,T]$ with 
\begin{align*}
    t_{k}=t_{k,0} < \dots < t_{k,n} < \dots < t_{k,N}=t_{k+1}, \ k=0, \dots, K-1, \ n=0, \dots N-1.
\end{align*}
The remaining discretization step of the original deep BSDE method \cite{E_2017} is to fix the time grid for some $N>0$ and approximate the time-continuous function $v$ with $\overline{v}= (\overline{v}_n)_{n=0}^{N-1}$, $\overline{v}_i\colon \R^d\times \mathbb{O}_k\to \R^d$, defined at the discrete time points $(t_{k,n})_{n=0}^{N-1}$, $k=0,\dots,K-1$, such that $\overline{v}_n\approx v(t_{k,n})$. We obtain the following discrete problem
\begin{align} \label{eq: discrete BSDE minimization}
    &\hspace{0em}
    \underset{
    \substack{
    w\in 
    L^\infty(\mathbb{O}_k;
    C(\R^d
    ;\R))
    \\
    (\overline{v}_n)_{n=0}^{N-1}\in 
    \prod_{n=0}^{N-1}
    L^\infty(\mathbb{O}_k;
    C(\R^d
    ;\R^d))
    }}
    {\mathrm{min}}\  
    \E
    \Bigg[
    \Big|
    \mathcal{Y}_{N}^{k}
    -
    \overline{g}_k
    (\mathcal{X}_{N}^k,
    O_{1:k})
    \Big|^2
    \Bigg]
    \\
\begin{split} \notag
    \mathcal{X}_{n+1}^k
    &=
    \mathcal{X}_{n}^k
    +
    b
    (\mathcal{X}_{n}^k)
    (t_{k,n+1}-t_{k,n})
    +
    \sigma
    (\mathcal{X}_{n}^k)
    (W_{t_{k,n+1}}-W_{t_{k,n}})
    ,\quad
    n=0,\dots,N-1,
    \\[1.0\jot]
    \mathcal{Z}_{n}^{k}
    &=
    \overline{v}_{n}
    (\mathcal{X}_{n}^k, O_{1:k})
    ,\hspace{18.5em}
    n=0,\dots,N-1,
    \\[-1.0\jot]
    \mathcal{Y}_{n+1}^{k}
    &=
    w(\mathcal{X}_{0}^k, O_{1:k})
    -
    \sum_{\ell=0}^n
    \Big(
    f_b
    \big(
    \mathcal{X}_{\ell}^k,
    \mathcal{Y}_{\ell}^{k},
    \mathcal{Z}_{\ell}^{k}
    \big)
    (t_{k,\ell+1}-t_{k,\ell})
    \\
    &\hspace{9.5em}-
    \big(\mathcal{Z}_{\ell}^{k}\big)^\top
    \sigma
    (\mathcal{X}_{\ell}^k)
    (W_{t_{k,\ell+1}}-W_{t_{k,\ell}})
    \Big)
    ,\quad
    n=0,\dots,N-1,
\end{split}
\end{align}
where $\overline{g}_k$ is defined, for $x\in\R^d$ and $o_{1:k}\in\mathbb{O}_k$, from the previous optimum $w_{k-1}^*$
\begin{align} 
\label{eq: overline-g_k}
    \overline{g}_k(x,o_{1:k})
    &=
    \begin{cases}
        w_{k-1}^*
        (x,o_{1:k-1})
        L(o_k,x),
        & 
        k\geq 1,
        \\
        \pi_0(x),
        & 
        k=0.
    \end{cases} 
\end{align}
We define our filter approximation $\widehat{p} = (\widehat{p}_k)_{k=1}^K$, of the exact filter $(p_k(t_k))_{k=1}^K$, by
\begin{align} \label{eq: deep bsde filter}
    \widehat{p}_k
    (x,o_{1:k})
    &=
    w_{k-1}^*
    (x,o_{1:k-1})
    L(o_k,x)
    ,\quad
    o_{1:k}\in \mathbb{O}_k
    ,\, x\in \R^d,
\end{align}
where $(w_k^*)_{k=0}^{K-1}$ are the solutions obtained from the discrete formulation \eqref{eq: discrete BSDE minimization}. See Figure~\ref{fig: graphical abstract} for a flowchart of the methodology to obtain $\widehat{p}$.

\begin{figure}[h]
\scalebox{0.70}{\input{include/new_figures/dbsde_GA}}
\vspace{0pt}
\caption{The figure illustrates the flowchart of the chain of prediction and update steps throughout the method. Each Deep BSDE box represents a BSDE which is solved through the optimization problem \eqref{eq: discrete BSDE minimization}.} 
\label{fig: graphical abstract}
\end{figure}
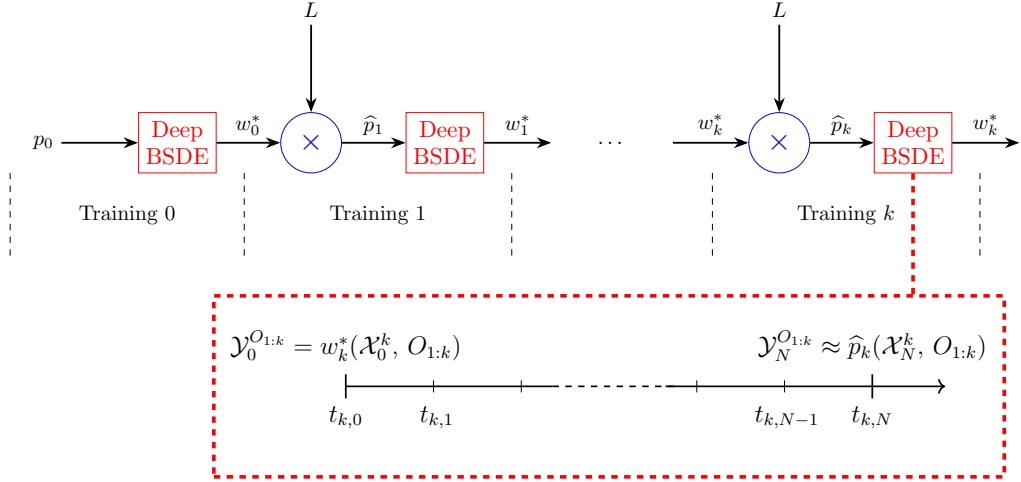

\section{Error analysis}
\label{section: error analysis}
In this section, we quantify how discretization and optimization errors propagate through the recursive filtering scheme by studying a strong error. To simplify the notation, we fix the time grid to satisfy, for all $k=0,\dots,K-1$, $n=0,\dots,N-1$,  $t_{k,n+1}-t_{k,n} = \frac{T}{KN} \eqdef \tau$. We begin Section~\ref{section: useful lemmas} by reciting some results that we need. In Section~\ref{section: theoretical convergence proof} we state the main result and give the corresponding proof. Throughout the section, we assume \eqref{eq: condition smooth} and \eqref{eq: condition hörmander}.

\subsection{Relevant inequalities} 
\label{section: useful lemmas}
Let $(X,Y,Z)$ be the solution to \eqref{theorem eq: BSDE cont formulation} and $(\widetilde{X},\widetilde{Y},\widetilde{Z}) = (\widetilde{X}^k,\widetilde{Y}^k,\widetilde{Z}^k)_{k=0}^{K-1}$ be the solution, for $k\in\{0,\dots,K-1\}$, $o_{1:k}\in \mathbb{O}_k$, $t\in[t_k,t_{k+1}]$, to 
\begin{align}
\begin{split} \label{eq: auxiliary BSDE formulation}
    \widetilde{X}_t^k
    &=
    \widetilde{X}_{t_k}^k
    +
    \int_{t_k}^t
    b
    (\widetilde{X}_s^k)
    \dd s
    +
    \int_{t_k}^t
    \sigma
    (\widetilde{X}_s^k)
    \dd W_s,
    \\
    \widetilde{Y}_t^k
    &=
    \overline{g}_k
    (\widetilde{X}_{t_{k+1}}^k
    ,
    o_{1:k})
    +
    \int_t^{t_{k+1}}
    f_b(
    \widetilde{X}_s^k,
    \widetilde{Y}_s^k,
    \widetilde{Z}_s^k
    )
    \dd s
    -
    \int_t^{t_{k+1}}
    \big(\widetilde{Z}_s^k\big)^\top
    \sigma(\widetilde{X}_s^k)
    \dd W_s. 
\end{split}
\end{align}
This system of equations is the same as \eqref{theorem eq: BSDE cont formulation} except for the terminal conditions of the $\widetilde{Y}^k$ processes and we note that $X=\widetilde{X}$. By the \emph{a priori} estimate \cite[Proposition~2.1]{BSDE_in_finance}, there exists a constant $C_1>0$ so that
\begin{align}
\label{eq: a priori BSDE Y}
    \sup_{t\in [t_k,t_{k+1}]}
    \E
    \Big[
    \big|
    Y_t^k
    -
    \widetilde{Y}_t^k
    \big|^2
    \Big]^{\frac{1}{2}}
    &\leq
    C_1
    \E
    \Big[
    \big|
    g_k(
    X_{t_{k+1}}^k,
    o_{1:k}
    )
    -
    \overline{g}_k(
    X_{t_{k+1}}^k
    ,
    o_{1:k}
    )
    \big|^2
    \Big]^{\frac{1}{2}}
    .
\end{align}
We remark that in \cite{BSDE_in_finance} the $Z^k$-component (and identically $\widetilde{Z}^k$) is defined such that the diffusion coefficient $\sigma(X^k)$ is absorbed into $Z^k$. Consequently, the corresponding BSDE in \cite{BSDE_in_finance} is written in a form that differs from \eqref{eq: auxiliary BSDE formulation}
by this convention. When adapting the \emph{a priori} estimates from \cite{BSDE_in_finance} to our setting, the process denoted by $Z^\top$ in \cite{BSDE_in_finance} should therefore be identified with $(Z^k)^\top \sigma(X^k)$ in our notation.

In \cite{han_convergence} the authors study the convergence of the deep BSDE method for more general coupled FBSDEs. The most restrictive condition in \cite{han_convergence} is a weak coupling condition. This condition is trivially satisfied in our uncoupled framework and more generally it is easy to verify that conditions \eqref{eq: condition smooth} and \eqref{eq: condition hörmander} are sufficient for \cite[Theorem~1]{han_convergence}. The result states that, up to a constant, the approximation error is bounded by the uniform time step $\tau$ and the error at the terminal condition, thus a hybrid \emph{a priori} and \emph{a posteriori} error bound. To state the result, applied to our setting, we introduce some notation. For $k\in\{0,\dots,K-1\}$ we define the piecewise constant processes
\begin{align*}
    (
    \mathcal{X}_{t}^k,
    \mathcal{Y}_{t}^k,
    \mathcal{Z}_{t}^k
    )
    =
    (
    \mathcal{X}_{n}^k,
    \mathcal{Y}_{n}^k,
    \mathcal{Z}_{n}^k
    )
    ,\quad 
    t\in[t_{k,n},t_{k,n+1}), 
    \, 
    n\in \{0,\dots,N-1\},
\end{align*}
where $(\mathcal{X}^k,\mathcal{Y}^k,\mathcal{Z}^k)$ are the solutions to \eqref{eq: discrete BSDE minimization}. According to \cite[Theorem~1]{han_convergence}, there exists a constant $C$ independent of $\tau$ such that
\begin{align} 
    \label{theorem eq: deep bsde bound}
    &\sup_{t \in [t_k,t_{k+1}]}
    \Big(
    \E 
    \big\| 
    \widetilde{X}_t^k 
    - 
    \mathcal{X}^k_t 
    \big\|^2 
    + 
    \E 
    \big| 
    \widetilde{Y}_t^k 
    - 
    \mathcal{Y}^k_t 
    \big|^2 
    \Big)
    \leq 
    C 
    \Big(
    \tau
    + 
    \E
    \big| 
    \overline{g}_k(
    \mathcal{X}^k_{N},o_{1:k}
    ) 
    - 
    \mathcal{Y}^k_N 
    \big|^2 
    \Big).
\end{align}
From \eqref{theorem eq: deep bsde bound} there exists $C_2>0$ such that for all $k\in\{0,\dots,K-1\}$, $o_{1:k}\in\mathbb{O}_k$, we have
\begin{align} 
\label{collorary eq: deep bsde bound square root Y}
    \sup_{t \in [t_k,t_{k+1}]}
    \E 
    \Big[
    \big| 
    \widetilde{Y}_t^k 
    - 
    \mathcal{Y}^k_t 
    \big|^2 
    \Big]^\frac{1}{2}
    &\leq
    C_2 
    \Bigg( 
    \tau^\frac{1}{2}
    + 
    \E
    \Big[
    \big| 
    \overline{g}_k(
    \mathcal{X}^k_{N},
    o_{1:k}
    ) 
    - 
    \mathcal{Y}^k_N 
    \big|^2 
    \Big]^\frac{1}{2}
    \Bigg).
\end{align}
We note that \cite{han_convergence} adopts the same convention for the process
$Z$ as \cite{BSDE_in_finance}, so the same adaptation applies.
\subsection{Convergence}
\label{section: theoretical convergence proof}
We are now ready to state our main result, building on the \emph{a priori} estimate \eqref{eq: a priori BSDE Y} and error bound \eqref{collorary eq: deep bsde bound square root Y}. Similar to \eqref{theorem eq: deep bsde bound} our theorem is a hybrid \emph{a priori} and \emph{a posteriori} bound. The \emph{a posteriori} term is directly related to the objective function in \eqref{eq: discrete BSDE minimization}.
\begin{theorem}
\label{main theorem: error bound}
    Let $p_k$, $k=1,\dots,K$, be the solution to \eqref{eq: global Fokker--Planck with update}--\eqref{eq: FP-update}, $\widehat{p}_k$, $k=1,\dots,K$, be given by \eqref{eq: deep bsde filter}, $\mathcal{X}^k$ and $\mathcal{Y}^k$ be defined as in \eqref{eq: discrete BSDE minimization}. If $b$ satisfies \eqref{eq: mu-bar choice}, then there exists a constant $C>0$, depending on $f,~T,~K,~C_1,~C_2$ and $L$, such that
    \begin{align}
    \begin{split}
        \label{main theorem eq: error bound}
        &\hspace{0.0em}
        \sup_{k\in\{1,\dots,K\}}
        \big\|
        p_{k}(t_k)
        -
        \widehat{p}_{k}
        \big\|_{L^\infty(\mathbb{O}_k;L^\infty(\R^d;\R))}
        \\
        &\hspace{7.5em}
        \leq
        C
        \Bigg(
        \tau^{\frac{1}{2}}
        +
        \sum_{k=0}^{K-1} 
        \sup_{o_{1:k}\in\mathbb{O}_k}
        \sup_{x \in \R^d}
        \E 
        \Big[
        \big| 
        \overline{g}_{k}(
        \mathcal{X}_N^{k,x},
        o_{1:k}
        ) 
        - 
        \mathcal{Y}_N^{k,x}
        \big|^2 
        \Big]^{\frac{1}{2}}
        \Bigg)
        .
    \end{split}
    \end{align}
\end{theorem}
\begin{remark}
    An analog of Theorem~\ref{main theorem: error bound} can be obtained for a more general drift coefficient $b \in C_{\mathrm{b}}^\infty(\R^d;\R^d)$ under the stronger ''ellipticity'' condition that $a=\sigma\sigma^\top$ is invertible. This can be proven by following an analogue to the proof presented below, but requires ellipticity to use a corresponding \emph{a priori} estimates from \cite[Proposition~2.1]{BSDE_in_finance} to the $Z$ process.
\end{remark}

\begin{proof}[Proof of Theorem~\ref{main theorem: error bound}]
We begin by fixing $k\in\{0,\dots,K-1\}$ and take $o_{1:k+1}\in \mathbb{O}_{k+1}$. 
Throughout the proof, $C>0$ denotes a generic constant whose value may change from line to line. 
The constant $C$ may depend on $f$, $T$, $K$, $C_1$, $C_2$, and the uniform bound on the likelihood $L$, but is independent of $\tau$, $x$, and the observations $o_{1:K}\in\mathbb{O}_K$. We recall that the exact filter solution is obtained by multiplying the prediction density with the likelihood
\begin{align*}
    p_{k+1}
    (t_{k+1},x,o_{1:k+1})
    &=
    p_k
    (t_{k+1},x,o_{1:k})
    L(o_{k+1},x).
\end{align*}
Analogously, we recall the approximation, $\widehat{p}$, from \eqref{eq: deep bsde filter}, defined by
\begin{align*}
    \widehat{p}_{k+1}(x,o_{1:k+1})
    &=
    w_k(x,o_{1:k})
    L(o_{k+1},x).
\end{align*}
Recalling the uniform bounding constant on $L$, we get
\begin{align}
\begin{split} 
\label{proof eq: initial LHS}
    &\hspace{-4em}
    \sup_{x\in\R^d}
    \big|
    p_{k+1}
    (t_{k+1},x,o_{1:k+1})
    -
    \widehat{p}_{k+1}(x,o_{1:k+1})
    \big|
    \\
    &=
    \sup_{x\in\R^d}
    \big|
    p_k
    (t_{k+1},x,o_{1:k})
    L(o_{k+1},x)
    -
    w_k(x,o_{1:k})
    L(o_{k+1},x)
    \big|
    \\
    &\leq
    C
    \sup_{x\in\R^d}
    \big|
    p_k
    (t_{k+1},x,o_{1:k})
    -
    w_k(x,o_{1:k})
    \big|.
\end{split}
\end{align}
To analyze this difference, we recall the Feynman--Kac formulation of $p$ from Corollary~\ref{corollary: feynman--kac formula} and the simplified $f$ in \eqref{eq: simplified f} which follows by \eqref{eq: mu-bar choice}. We replace $t_{k+1}$ by $t_{k,N}$ and $t_{k}$ by $t_{k,0}$ to obtain
\begin{align*}
    p_k
    (t_{k+1},x)
    &=
    \E
    \Bigg[
    g_k(
    X_{t_{k,N}}^{k,x}
    )
    +
    \int_{t_{k,0}}^{t_{k,N}}
    f(
    X_s^{k,x},
    Y_s^{k,x}
    )
    \dd s
    \Bigg].
\end{align*}
For convenience, we have dropped the $o_{1:k}$ and $o_{k+1}$ from the notation and will continue to do so whenever possible. Likewise, we see that the approximation $w_k$, defined by \eqref{eq: discrete BSDE minimization}, satisfies 
\begin{align*}
    w_k(x)
    &=
    \E
    \big[
    w_k(\mathcal{X}_{0}^k)
    \mid
    \mathcal{X}_{0}^k
    =
    x
    \big]
    \\
    &=
    \E
    \Bigg[
    \mathcal{Y}_N^{k,x}
    +
    \sum_{\ell=0}^{N-1}
    \Big(
    f(
    \mathcal{X}_\ell^{k,x},
    \mathcal{Y}_\ell^{k,x}
    )
    (t_{k,\ell+1}-t_{k,\ell})
    +
    \big(\mathcal{Z}_\ell^{k,x}\big)^\top
    \sigma(\mathcal{X}_\ell^{k,x})
    (W_{t_{k,\ell+1}}
    -
    W_{t_{k,\ell}})
    \Big)
    \Bigg]
    \\
    &=
    \E
    \Bigg[
    \mathcal{Y}_N^{k,x}
    +
    \sum_{\ell=0}^{N-1}
    f(
    \mathcal{X}_\ell^{k,x},
    \mathcal{Y}_\ell^{k,x}
    )
    (t_{k,\ell+1}-t_{k,\ell})
    \Bigg].
\end{align*}
Looking at the difference, rearranging the terms and applying the triangle inequality, we see that 
\begin{align*}
    &\hspace{-1em}
    \big|
    p_k
    (t_{k+1}, x) 
    - 
    w_k
    (x)
    \big|
    \\
    &\hspace{1em}=
    \bigg|
    \E
    \Bigg[
    g_k(
    X_{t_{k,N}}^{k,x}
    )
    +
    \int_{t_{k,0}}^{t_{k,N}}
    f(
    X_s^{k,x},
    Y_s^{k,x}
    )
    \dd s
    \Bigg]
    -
    \E
    \Bigg[
    \mathcal{Y}_N^{k,x}
    +
    \sum_{\ell=0}^{N-1}
    f(
    \mathcal{X}_\ell^{k,x},
    \mathcal{Y}_\ell^{k,x}
    )
    (t_{k,\ell+1}-t_{k,\ell})
    \Bigg]
    \bigg|
    \\
    &\hspace{1em}\leq
    \Big|
    \E
    \Big[
    g_k(
    X_{t_{k,N}}^{k,x}
    )
    -
    \mathcal{Y}_N^{k,x}
    \Big]
    \Big|
    +
    \bigg|
    \sum_{\ell=0}^{N-1}
    \E
    \bigg[
    \int_{t_{k,\ell}}^{t_{k,\ell+1}}
    f(
    X_s^{k,x},
    Y_s^{k,x}
    )
    -
    f(
    \mathcal{X}_\ell^{k,x},
    \mathcal{Y}_\ell^{k,x}
    )
    \dd s
    \bigg]
    \bigg|
    \\
    &\hspace{1em}=
    \text{I}
    +
    \text{II}.
\end{align*}
We begin by looking at the first term I, in which we add and subtract the auxiliary solution $(\widetilde{Y}^{k})$ that satisfies $
\widetilde{Y}_{t_{k,N}}^{k,x}
= 
\overline{g}_k(
X_{t_{k,N}}^{k,x}
)
$. 
Together with the triangle inequality, we obtain
\begin{align*}
    \text{I}
    &=
    \Big|
    \E
    \big[
    g_k(
    X_{t_{k,N}}^{k,x}
    )
    -
    \widetilde{Y}_{t_{k,N}}^{k,x}
    \big]
    +
    \E
    \big[
    \widetilde{Y}_{t_{k,N}}^{k,x}
    -
    \mathcal{Y}_N^{k,x}
    \big]
    \Big|
    \\
    &\leq
    \E
    \Big[
    \big|
    g_k(
    X_{t_{k,N}}^{k,x}
    )
    -
    \overline{g}_k(
    X_{t_{k,N}}^{k,x}
    )
    \big|
    \Big]
    +
    \E
    \Big[
    \big|
    \widetilde{Y}_{t_{k,N}}^{k,x}
    -
    \mathcal{Y}_N^{k,x}
    \big|
    \Big]
    \\
    &=
    \text{I}_{1}
    +
    \text{I}_{2}.
\end{align*}
For the first term, we recall $g_k$ from \eqref{theorem eq: g_k} and $\overline{g}_k$ from \eqref{eq: overline-g_k}. By applying Hölder's inequality and \eqref{eq: uniform bound L} we get
\begin{align*}
    \text{I}_1
    &=
    \E
    \Big[
    \big|
    p_{k-1}(
    t_k,
    X_{t_{k,N}}^{k,x}
    )
    L(o_k,X_{t_{k,N}}^{k,x})
    -
    w_{k-1}(
    X_{t_{k,N}}^{k,x}
    )
    L(o_k,X_{t_{k,N}}^{k,x})
    \big|
    \Big]
    \\
    &\leq
    C\,
    \E
    \Big[
    \big|
    p_{k-1}(
    t_k,
    X_{t_{k,N}}^{k,x}
    )
    -
    w_{k-1}(
    X_{t_{k,N}}^{k,x}
    )
    \big|
    \Big]
    \\
    &\leq
    C
    \sup_{x\in\R^d}
    \big|
    p_{k-1}(
    t_k,
    x
    )
    -
    w_{k-1}(
    x
    )
    \big|.
\end{align*}
This is the error from the previous time step multiplied with a constant.
For the term $\text{I}_2$ we apply the Cauchy--Schwarz inequality and \eqref{collorary eq: deep bsde bound square root Y} to obtain
\begin{align*}
    \text{I}_2
    &\leq
    \E
    \Big[
    \big|
    \widetilde{Y}_{t_{k,N}}^{k,x}
    -
    \mathcal{Y}_N^{k,x}
    \big|^2
    \Big]^{\frac{1}{2}}
    \leq
    C_2 
    \Bigg( 
    \tau^\frac{1}{2}
    + 
    \E
    \Big[
    \big| 
    \overline{g}_k(
    \mathcal{X}^{k,x}_{N}
    ) 
    - 
    \mathcal{Y}^{k,x}_N 
    \big|^2 
    \Big]^\frac{1}{2}
    \Bigg).
\end{align*}
By collecting $\text{I}_1$ and $\text{I}_2$, we get
\begin{align*}
    \text{I}
    &\leq
    C
    \Big(
    \sup_{x\in\R^d}
    \big|
    p_{k-1}(
    t_k,
    x
    )
    -
    w_{k-1}(
    x
    )
    \big|
    +
    \tau^\frac{1}{2}
    + 
    \E
    \Big[
    \big| 
    \overline{g}_k(
    \mathcal{X}^{k,x}_{N}
    ) 
    - 
    \mathcal{Y}^{k,x}_N 
    \big|^2 
    \Big]^\frac{1}{2}
    \Big)
    .
\end{align*}
Next we rewrite the term $\text{II}$ by applying the Fubini--Tonelli theorem and using the triangle inequality to obtain
\begin{align*}
    \text{II}
    &=
    \bigg|
    \sum_{\ell=0}^{N-1}
    \E
    \bigg[
    \int_{t_{k,\ell}}^{t_{k,\ell+1}}
    f(
    X_s^{k,x},
    Y_s^{k,x}
    )
    -
    f(
    \mathcal{X}_\ell^{k,x},
    \mathcal{Y}_\ell^{k,x}
    )
    \dd s
    \bigg]
    \bigg|
    \notag
    \\
    &\leq
    \sum_{\ell=0}^{N-1}
    \int_{t_{k,\ell}}^{t_{k,\ell+1}}
    \E
    \Big[
    \big|
    f(
    X_s^{k,x},
    Y_s^{k,x}
    )
    -
    f(
    \mathcal{X}_\ell^{k,x},
    \mathcal{Y}_\ell^{k,x}
    )
    \big|
    \Big]
    \dd s
    .
\end{align*}
The absolute integrability of the integrand, required for the use of the Fubini--Tonelli theorem, can be shown by the properties of $f$ and by standard moment estimates. 
We continue by using the Lipschitz continuity of $f$ and applying the Cauchy--Schwarz inequality
\begin{align}
    \text{II}
    &\leq
    C
    \sum_{\ell=0}^{N-1}
    \int_{t_{k,\ell}}^{t_{k,\ell+1}}
    \bigg(
    \E
    \Big[
    \big\|
    X_s^{k,x}
    -
    \mathcal{X}_\ell^{k,x}
    \big\|
    +
    \big|
    Y_s^{k,x}
    -
    \mathcal{Y}_\ell^{k,x}
    \big|
    \Big]
    \bigg)
    \dd s
    \notag
    \\
    &\leq
    C
    \sum_{\ell=0}^{N-1}
    \int_{t_{k,\ell}}^{t_{k,\ell+1}}
    \bigg(
    \E
    \Big[
    \big\|
    X_s^{k,x}
    -
    \mathcal{X}_\ell^{k,x}
    \big\|^2
    \Big]^{\frac{1}{2}}
    +
    \E
    \Big[
    \big|
    Y_s^{k,x}
    -
    \mathcal{Y}_\ell^{k,x}
    \big|^2
    \Big]^\frac{1}{2}
    \bigg)
    \dd s 
    .
    \label{proof eq: II RHS 2}
\end{align}
The first term in the integral, is the strong error of the Euler--Maruyama approximation of a time homogeneous SDE, and satisfies order $\frac{1}{2}$ convergence in $\tau$. For the second term, we add and subtract the auxiliary process $\widetilde{Y}^k$, defined as the solution to \eqref{eq: auxiliary BSDE formulation}, inside the norm.  Applying the triangle inequality and taking the supremum over $[t_k,t_{k+1}]$ yields
\begin{align*}
    \E
    \Big[
    \big|
    Y_s^{k,x}
    -
    \mathcal{Y}_\ell^{k,x}
    \big|^2
    \Big]^\frac{1}{2}
    &\leq
    \E
    \Big[
    \big|
    Y_s^{k,x}
    -
    \widetilde{Y}_s^{k,x}
    \big|^2
    \Big]^\frac{1}{2}
    +
    \E
    \Big[
    \big|
    \widetilde{Y}_s^{k,x}
    -
    \mathcal{Y}_\ell^{k,x}
    \big|^2
    \Big]^\frac{1}{2}
    \\
    &
    \leq
    \sup_{t\in[t_k,t_{k+1}]}
    \E
    \Big[
    \big|
    Y_t^{k,x}
    -
    \widetilde{Y}_t^{k,x}
    \big|^2
    \Big]^\frac{1}{2}
    +
    \sup_{t\in[t_k,t_{k+1}]}
    \E
    \Big[
    \big|
    \widetilde{Y}_t^{k,x}
    -
    \mathcal{Y}_t^{k,x}
    \big|^2
    \Big]^\frac{1}{2}
    .
\end{align*}
By applying \eqref{eq: a priori BSDE Y} and \eqref{collorary eq: deep bsde bound square root Y} for the respective term, we get
\begin{align*}
    &\E
    \Big[
    \big|
    Y_s^{k,x}
    -
    \mathcal{Y}_\ell^{k,x}
    \big|^2
    \Big]^\frac{1}{2}
    \leq
    C
    \bigg(
    \E
    \Big[
    \big|
    g_k(
    X_{t_{k,N}}^{k,x}
    )
    -
    \overline{g}_k(
    X_{t_{k,N}}^{k,x}
    )
    \big|^2
    \Big]^{\frac{1}{2}}
    +
    \tau^\frac{1}{2}
    +
    \E 
    \Big[
    \big| 
    \overline{g}_{k}(\mathcal{X}_N^{k,x}) 
    - 
    \mathcal{Y}_N^{k,x}
    \big|^2
    \Big]^{\frac{1}{2}}
    \bigg)
    .
\end{align*}
Inserting this bound into \eqref{proof eq: II RHS 2}, we obtain
\begin{align*}
    &\hspace{-2em}\text{II}
    \leq
    C
    \sum_{\ell=0}^{N-1}
    \int_{t_{k,\ell}}^{t_{k,\ell+1}}
    \bigg(
    \tau^{\frac{1}{2}}
    + 
    \E
    \Big[
    \big|
    g_k(
    X_{t_{k,N}}^{k,x}
    )
    -
    \overline{g}_k(
    X_{t_{k,N}}^{k,x}
    )
    \big|^2
    \Big]^{\frac{1}{2}}
    +
    \E 
    \Big[
    \big| 
    \overline{g}_{k}(\mathcal{X}_N^{k,x}) 
    - 
    \mathcal{Y}_N^{k,x}
    \big|^2
    \Big]^{\frac{1}{2}}
    \bigg)
    \dd s
    \\
    &\hspace{-1em}
    \leq
    C
    \frac{T}{K}
    \bigg(
    \tau^{\frac{1}{2}}
    + 
    \E
    \Big[
    \big|
    g_k(
    X_{t_{k,N}}^{k,x}
    )
    -
    \overline{g}_k(
    X_{t_{k,N}}^{k,x}
    )
    \big|^2
    \Big]^{\frac{1}{2}}
    +
    \E 
    \Big[
    \big| 
    \overline{g}_{k}(\mathcal{X}_N^{k,x}) 
    - 
    \mathcal{Y}_N^{k,x}
    \big|^2
    \Big]^{\frac{1}{2}}
    \bigg)
    \dd s
    .
\end{align*}
Adapting the same arguments as for $\text{I}_1$ on the second term, we have
\begin{align*}
    \text{II}
    &\leq
    C
    \bigg(
    \tau^{\frac{1}{2}}
    + 
    \sup_{x\in\R^d}
    \big|
    p_{k-1}(
    x
    )
    -
    w_{k-1}(
    x
    )
    \big|
    +
    \E 
    \Big[
    \big| 
        \overline{g}_{k}(\mathcal{X}_N^{k,x}) 
        - 
        \mathcal{Y}_N^{k,x}
    \big|^2 
    \Big]^{\frac{1}{2}}
    \bigg),
\end{align*}
which up to a constant is the same to the bound for $\text{I}$.
By defining the error term
$
e_k 
\defeq 
\sup_{x\in\R^d}
\big|
p_k
(t_{k+1}, x) 
- 
w_k
(x)
\big|
$ 
and combining the bounds of $\text{I}$ and $\text{II}$ we get for some constant $C>0$ depending on $f$, $T$, $K$, $C_1$, $C_2$, and $L$
\begin{align*}
    &\big|
    p_k
    (t_{k+1}, x) 
    - 
    w_k
    (x)
    \big|
    \leq
    C
    e_{k-1}
    +
    C
    \bigg(
    \tau^{\frac{1}{2}}
    +
    \E 
    \Big[
    \big| 
        \overline{g}_{k}(\mathcal{X}_N^{k,x}) 
        - 
        \mathcal{Y}_N^{k,x}
    \big|^2 
    \Big]^{\frac{1}{2}}
    \bigg)
    .
\end{align*}
Taking supremum over $x\in\R^d$ on the right-hand side, followed by the left-hand side, we obtain the recursive bound
\begin{align*}
    e_k
    &\leq
    C
    e_{k-1}
    +
    C
    \bigg(
    \tau^{\frac{1}{2}}
    +
    \sup_{x\in\R^d}
    \E 
    \Big[
    \big| 
        \overline{g}_{k}(\mathcal{X}_N^{k,x}) 
        - 
        \mathcal{Y}_N^{k,x}
    \big|^2 
    \Big]^{\frac{1}{2}}
    \bigg).
\end{align*}
Applying this bound recursively, we get
\begin{align*}
    e_k 
    &\leq 
    \sum_{j=0}^k 
    C^{k-j+1}
    \bigg(
    \tau^{\frac{1}{2}}
    +
    \sup_{x\in\R^d}
    \E 
    \Big[
    \big| 
        \overline{g}_{j}(\mathcal{X}_N^{j,x}) 
        - 
        \mathcal{Y}_N^{j,x}
    \big|^2 
    \Big]^{\frac{1}{2}}
    \bigg)
    \\
    &\leq
    C
    \Bigg(
    \tau^{\frac{1}{2}}
    +
    \sum_{j=0}^{K-1} 
    \sup_{x \in \R^d}
    \E 
    \Big[
    \big| 
    \overline{g}_{j}(\mathcal{X}_N^{j,x}) 
    - 
    \mathcal{Y}_N^{j,x}
    \big|^2 
    \Big]^{\frac{1}{2}}
    \Bigg)
    ,
\end{align*}
for some constant $C$ depending on $f,~T,~K,~C_1,~C_2$ and $L$. We insert this bound into \eqref{proof eq: initial LHS} and finish by taking the supremum over $o_{1:K}\in\mathbb{O}_K$ on the right-hand side, moving it inside the sum and finally taking the supremum on the left-hand side to obtain
\begin{align*}
    &\big\|
    p_{k+1}(t_{k+1})
    -
    \widehat{p}_{k+1}
    \big\|_{L^\infty(\mathbb{O}_k;L^\infty(\R^d;\R))}
    \leq
    C
    \Bigg(
    \tau^{\frac{1}{2}}
    +
    \sum_{k=0}^{K-1} 
    \sup_{o_k\in\mathbb{O}_k}
    \sup_{x \in \R^d}
    \E 
    \Big[
    \big| 
    \overline{g}_{k}(
    \mathcal{X}_N^{k,x},
    o_{1:k}
    ) 
    - 
    \mathcal{Y}_N^{k,x}
    \big|^2 
    \Big]^{\frac{1}{2}}
    \Bigg)
    .
\end{align*}
Since the bound holds for all $k \in \{0, \dots, K-1\}$, the proof is complete.
\end{proof}

\section{Numerical results}
\label{section: numerical results}
In this section, we investigate the convergence result of Theorem~\ref{main theorem: error bound} numerically. To this end we begin in Section~\ref{section: numerical approximation} by outlining the final approximation steps for an employable algorithm. Section~\ref{section: numerical experiments} contains the two numerical examples in which we evaluate the error terms of the bound \eqref{main theorem eq: error bound} numerically.
\subsection{Numerical approximation}
\label{section: numerical approximation}
In \eqref{eq: discrete BSDE minimization}, we defined the optimization problem over all continuous functions, which is not feasible in practice.
Instead, we consider classes of neural networks denoted $\mathbf{NN}^{\Theta,p,k} \subset L^\infty(\mathbb{O}_k;C(\R^d; \R^p))$ for $(p,k)\in\{1,d\}\times\{0,1,\dots,K-1\}$. The specifications of these classes, and the set $\Theta$ of learnable parameters, are presented in Appendix~\ref{Appendix: implementation details}. By parameterizing the continuous functions $(w,\overline{v})$ using neural networks $(w^\theta,\overline{v}^\theta)\in \mathbf{NN}^{\Theta,1,k}\times \mathbf{NN}^{\Theta,d,k}$ and optimizing over their weights $\theta$, the scheme becomes fully implementable. This approach leverages the expressive power of neural networks, enabling a flexible and efficient way of approximating the filtering density. The new recursive scheme is given, for $k=0,\dots,K-1$, by
\begin{align} \label{eq: NN minimization}
    &\hspace{0em}
    \underset{
    \substack{
    w^\theta\in 
    \mathbf{NN}^{\Theta,1,k}
    \\
    (\overline{v}_n^\theta)_{n=0}^{N-1}\in 
    \prod_{n=0}^{N-1}
    \mathbf{NN}^{\Theta,d,k}
    }}
    {\mathrm{min}}\  
    \E
    \Bigg[
    \Big|
    \mathcal{Y}_{N}^{k}
    -
    \overline{g}_k
    (\mathcal{X}_{N}^k,
    O_{1:k})
    \Big|^2
    \Bigg]
    \\
\begin{split} \notag
    \mathcal{X}_{n+1}^k
    &=
    \mathcal{X}_{n}^k
    +
    b
    (\mathcal{X}_{n}^k)
    (t_{k,n+1}-t_{k,n})
    +
    \sigma
    (\mathcal{X}_{n}^k)
    (W_{t_{k,n+1}}-W_{t_{k,n}})
    ,\quad
    n=0,\dots,N-1,
    \\[1.0\jot]
    \mathcal{Z}_{n}^{k}
    &=
    \overline{v}_{n}^\theta
    (\mathcal{X}_{n}^k, O_{1:k})
    ,\hspace{18.5em}
    n=0,\dots,N-1,
    \\[-1.0\jot]
    \mathcal{Y}_{n+1}^{k}
    &=
    w^\theta
    (\mathcal{X}_{0}^k, O_{1:k})
    -
    \sum_{\ell=0}^n
    \Big(
    f_b
    \big(
    \mathcal{X}_{\ell}^k,
    \mathcal{Y}_{\ell}^{k},
    \mathcal{Z}_{\ell}^{k}
    \big)
    (t_{k,\ell+1}-t_{k,\ell})
    \\
    &\hspace{9.5em}-
    \big(\mathcal{Z}_{\ell}^{k}\big)^\top
    \sigma
    (\mathcal{X}_{\ell}^k)
    (W_{t_{k,\ell+1}}-W_{t_{k,\ell}})
    \Big)
    ,\quad
    n=0,\dots,N-1.
\end{split}
\end{align}
Likewise, the expected value in the objective function cannot be evaluated in closed form and must instead be approximated via Monte Carlo sampling. To this end, we generate a set of $M_{\text{train}}$ independent samples $(S^{(m)}, O^{(m)}, X^{(m)}, W^{(m)})_{m=1}^{M_{\text{train}}}$ from the system \eqref{introduction eq: state}, \eqref{introduction eq: obs}, and the auxiliary process \eqref{eq: auxiliary process}. These samples define $M_{\text{train}}$ trajectories over which we evaluate the loss functional appearing in \eqref{eq: NN minimization}. The resulting empirical average used during training is defined as
\begin{align} \label{eq: Monte Carlo minimization}
    &\hspace{-4em}
    \underset{
    \substack{
    w^\theta\in 
    \mathbf{NN}^{\Theta,1,k}
    \\
    (\overline{v}_n^\theta)_{n=0}^{N-1}\in 
    \prod_{n=0}^{N-1}
    \mathbf{NN}^{\Theta,d,k}
    }}
    {\mathrm{min}}\  
    \frac{1}{M_{\text{train}}}
    \sum_{m=1}^{M_{\text{train}}}
    \Big|
    \mathcal{Y}_{N}^{k,(m)}
    -
    \overline{g}_k
    \big(
    \mathcal{X}_{N}^{k,(m)},
    O_{1:k}^{(m)}
    \big)
    \Big|^2
    .
\end{align}
A schematic overview of the method is presented in Figure~\ref{fig: schematic}.

\begin{figure}[h]
\scalebox{0.575}{\input{include/new_figures/obs_dbsde_schematic}}
\vspace{-10pt}
\caption{The figure illustrates the numerical scheme for a fixed $k=1,\dots,K-1$ in the optimization problem \eqref{eq: NN minimization}. The color of the arrows between the boxes indicate if the map consists of an Euler--Maruyama step, a parameterized neural network step or simply an input to a function.} 
\vspace{-10pt}
\label{fig: schematic}
\end{figure}
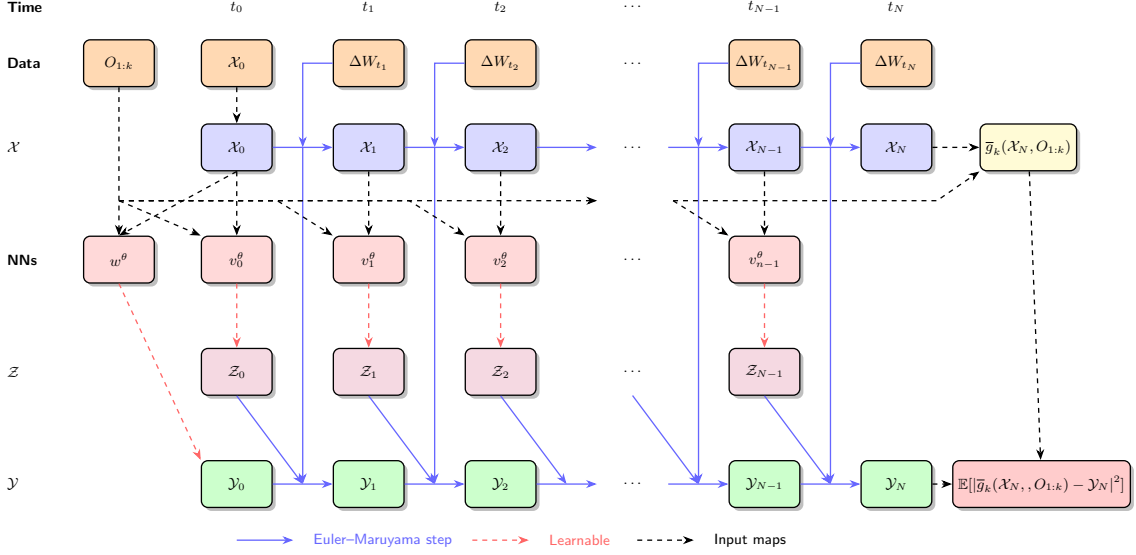

In the experiments in Section~\ref{section: numerical experiments} we use the normalized versions of $\overline{g}_k$, $k=1,\dots,K-1$, rather than the unnormalized counterparts. More precisely, we redefine the target function in \eqref{eq: Monte Carlo minimization}, for $k=1,\dots,K-1$, $o_{1:k}\in\mathbb{O}_k$, as 
\begin{align*}
    \overline{g}_k
    (x,
    o_{1:k})
    &=
    \frac{
    w_{k-1}^*
    (x,o_{1:k-1})
    L(o_k,x)
    }{
    C(o_{1:k})
    }
    ,\quad
    x\in\R^d,
    \\
    C(o_{1:k}) 
    &=
    \int_{\R^d}
    w_{k-1}^*
    (z,o_{1:k-1})
    L(o_k,z)
    \dd z.
\end{align*}
In the one-dimensional examples, the normalizing constant $C(O_{1:k}^{(m)})$ required for each training sample $m = 1, \dots, M_{\text{train}}$ is approximated using quadrature. Specifically, we define
\begin{align*}
    \widehat{C}^{(m)}
    &=
    \sum_{j=1}^{J}
    w_{k-1}^*
    (z_j,o_{1:k-1})
    L(o_k,z_j)
    \Delta z,
\end{align*}
where $(z_j)_{j=1}^J \subset \mathbb{R}$ is a fixed equidistant spatial grid with $\Delta z = z_{j} - z_{j-1}$ for $j = 2, \dots, J$. The grid is chosen to appropriately cover the support of the solution and to resolve the relevant features of the integrand. This quadrature-based approximation of the normalizing constant introduces an additional source of numerical error. To mitigate its effect, the number of grid points $J$ must be chosen sufficiently large to ensure accurate integration. Further implementation details are provided in Appendix~\ref{Appendix: implementation details}.

\subsection{Numerical experiments}
\label{section: numerical experiments}
In this section we empirically examine the numerical convergence of our approximate filter applied to \eqref{introduction eq: state}--\eqref{introduction eq: obs}. We test the method on two one-dimensional examples, one linear and one nonlinear. In both examples we set $T=1$, $K=10$, $d=d'=1$, $\sigma(x) = 1$, $L(o,x) = \N(o\mid h(x), R)$ with $h(x) = x$ and $R=1$. We note that this is a fairly low signal-to-noise ratio possibly making the filtering problem harder.

Recalling the error bound \eqref{main theorem eq: error bound} from Theorem~\ref{main theorem: error bound}, we now turn to estimating the quantities appearing in the bound. Specifically, we study convergence numerically by evaluating the error between $p_k$ and $\widehat{p}_k$ in the norm $L^\infty(\mathbb{O}_k; L^\infty(\R^d;\R))$, along with the corresponding \emph{a posteriori} term, across a range of discretizations. In the figures below, we denote these two quantities by
\begin{align*}
    e_k
    &=
    \|
    p_k(t_k)
    -
    \widehat{p}_k
    \|_{L^\infty(\mathbb{O}_k;L^\infty(\R^d;\R))}
    ,\quad
    k=1,\dots,K,
    \\
    E_k
    &=
    \sup_{o_{1:k}\in\mathbb{O}_k}
    \sup_{x\in\R^d}
    \E 
    \Big[
    \big| 
    \overline{g}_{k-1}(
    \mathcal{X}_N^{k-1,x},
    o_{1:k}
    ) 
    - 
    \mathcal{Y}_N^{k-1,x}
    \big|^2 
    \Big]^{\frac{1}{2}}
    ,\quad
    k=1,\dots,K.
\end{align*}
Here, $e_k$ measures the pointwise approximation error of the filtering density at time $t_k$, whereas $E_k$ reflects the residual error in the learning objective, serving as an \emph{a posteriori} error of how far the obtained solution is from the theoretical optimum at time $t_k$. See Appendix~\ref{Appendix: implementation details} for further details on the numerical approximation of these error terms.

\subsubsection*{Ornstein--Uhlenbeck process} The first equation we consider is a linear SDE with the drift coefficient $\mu(x) = -x$ and is solved by the Ornstein--Uhlenbeck process. In this linear setting, with a linear measurement function, the solution $p$ to \eqref{eq: global Fokker--Planck with update}--\eqref{eq: FP-update} is computed exactly by the Kalman filter. In this example, the coefficients satisfy conditions~\eqref{eq: condition smooth} and~\eqref{eq: condition hörmander}, and the convergence in Theorem~\ref{main theorem: error bound} applies. In Figure~\ref{fig: OU over time}, we report the pointwise errors $e_k$ and $E_k$, evaluated at each observation time $t_k = \frac{k}{K}$ for $k = 1, \dots, K$. The experiment is repeated for seven discretizations with $N = 2^j$, $j = 0, \dots, 6$, corresponding to the number of time steps between successive observation times. We note that both error terms decrease at all time steps as we increase the number of intermediate discretization steps. As expected, both error metrics decrease uniformly in time as $N$ increases. We also note that, for coarse discretizations, $e_k$ increases over time, indicating a cumulative error effect. However, for sufficiently fine resolutions ($N \geq 16$), the error remains nearly constant. The behavior of $E_k$ exhibits a similar trend, although not as clear.

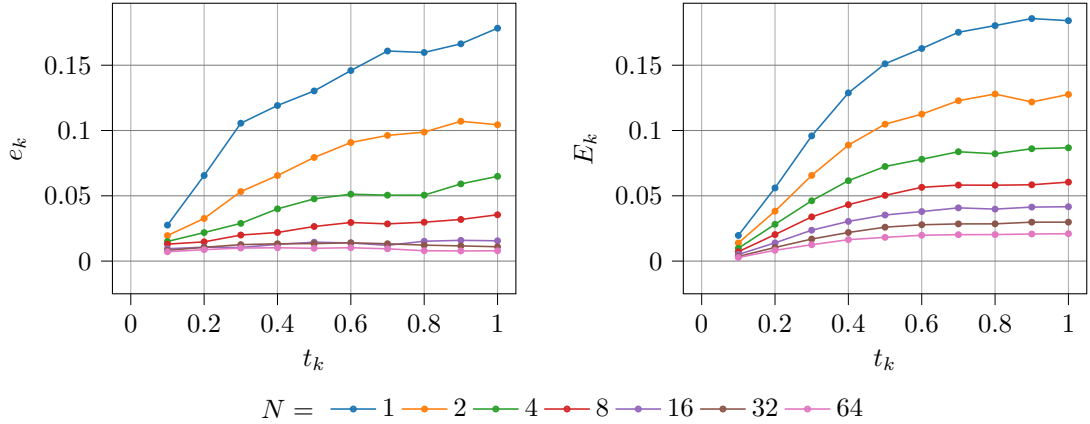
\begin{figure}[h]
\captionsetup{justification=raggedleft}
\begin{subfigure}[b]{0.49\linewidth}
\centering
\input{include/new_figures/OU_Linf}
\captionsetup{margin=90pt}
\vspace{0pt}
\end{subfigure}
\begin{subfigure}[b]{0.49\linewidth}
\centering
\input{include/new_figures/OU_Post}
\captionsetup{margin=80pt}
\vspace{0pt}
\end{subfigure}
\begin{tikzpicture}[overlay]
\node at (0,0) {$ $};
\node at (-5.3,-0.0) {\ref*{fig: over time legend OU}};
\node at (-9.4, -0.0) {$N = $};
\end{tikzpicture}
\vspace{-10pt}
\captionsetup{justification = justified}
\caption{Ornstein--Uhlenbeck process. The figure illustrates the error trajectories for seven different discretizations. To the left we see the error $e_k$ and to the right we see the residual $E_k$, $k=1,\dots,K$.}
\vspace{-4pt}
\label{fig: OU over time}
\end{figure}

To investigate the convergence rate, we fix $k=K$ and plot both the final time error $e_K$ and the accumulated \emph{a posteriori} term $E := \sum_{k=1}^K E_k$ in Figure~\ref{fig: OU conv plot}. These quantities are shown together with a reference slope of $N^{-\frac{1}{2}}$ on a logarithmic scale. We see that both terms converge at least with order $\frac{1}{2}$. These results are consistent with the theoretical error bound in \eqref{main theorem eq: error bound}, where $e_K$ is dominated by the discretization error $N^{-\frac{1}{2}}$ and the residual $E$.

\begin{figure}[h]
\captionsetup{justification=raggedleft}
\begin{subfigure}[b]{0.49\linewidth}
\centering
\input{include/new_figures/OU_Linf_loglog}
\captionsetup{margin=90pt}
\vspace{0pt}
\end{subfigure}
\begin{subfigure}[b]{0.49\linewidth}
\centering
\input{include/new_figures/OU_Post_loglog}
\captionsetup{margin=80pt}
\vspace{0pt}
\end{subfigure}
\begin{tikzpicture}[overlay]
\node at (0,0) {$ $};
\node at (-5.9,0.1) {\ref*{fig: conv legend OU}};
\end{tikzpicture}
\vspace{-10pt}
\captionsetup{justification = justified}
\caption{Ornstein--Uhlenbeck process. The figure presents the error and the accumulated residual at the final time for seven different discretizations together with a reference line of order $\frac{1}{2}$. To the left we have $e_K$ and to the right the accumulated residual $E=\sum_{k=1}^K E_k$.}
\vspace{-10pt}
\label{fig: OU conv plot}
\end{figure}
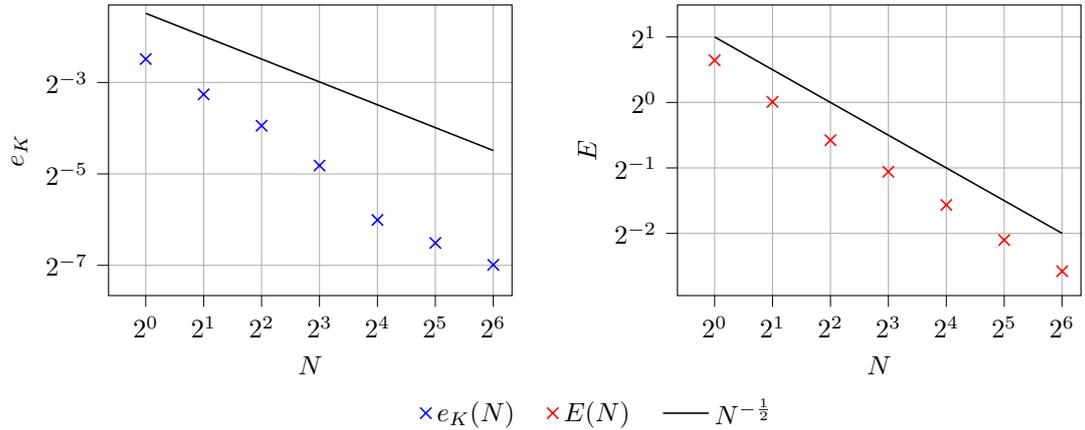

\subsubsection*{Bistable process}
The second example concerns an SDE with nonlinear drift given by $\mu(x) = \frac{2}{5}(5x - x^3)$. Here, condition~\eqref{eq: condition smooth} is not satisfied due to the cubic drift, and no theoretical convergence guarantee is available. The drift corresponds to the gradient of a double-well potential, leading to a bimodal invariant distribution for $S_t$ as $t \to \infty$. The dynamics  \eqref{introduction eq: state} of $S$, initialized from $\pi_0$, tend to drive the state $S$ toward the two stable modes of the potential landscape. In this nonlinear setting, no closed-form solution is available for the coupled system \eqref{eq: global Fokker--Planck with update}--\eqref{eq: FP-update}. As a result, we employ a particle filter to compute a reference solution; implementation details are provided in Appendix~\ref{Appendix: implementation details}. In Figure~\ref{fig: bistable over time}, we present the error measures $e_k$ and $E_k$ at each observation time $t_k$, for $k = 1, \dots, K$, across seven discretizations with $N = 2^j$, $j = 0, \dots, 6$. Similar to the linear case, the error decreases uniformly as $N$ increases. Moreover, for the finest discretizations ($N = 32$ and $N = 64$), the error appears to stabilize over time. 

\begin{figure}[h]
\captionsetup{justification=raggedleft}
\begin{subfigure}[b]{0.49\linewidth}
\centering
\input{include/new_figures/BI_Linf}
\captionsetup{margin=90pt}
\vspace{0pt}
\end{subfigure}
\begin{subfigure}[b]{0.49\linewidth}
\centering
\input{include/new_figures/BI_Post}
\captionsetup{margin=80pt}
\vspace{0pt}
\end{subfigure}
\begin{tikzpicture}[overlay]
\node at (0,0) {$ $};
\node at (-5.5,-0.4) {\ref*{fig: over time legend BI}};
\node at (-9.6, -0.4) {$N = $};
\end{tikzpicture}
\vspace{-2pt}
\captionsetup{justification = justified}
\caption{Bistable process. The figure illustrates the error trajectories for seven different discretizations. To the left we see the error $e_k$ and to the right we see the residual $E_k$, $k=1,\dots,K$.}
\vspace{-10pt}
\label{fig: bistable over time}
\end{figure}
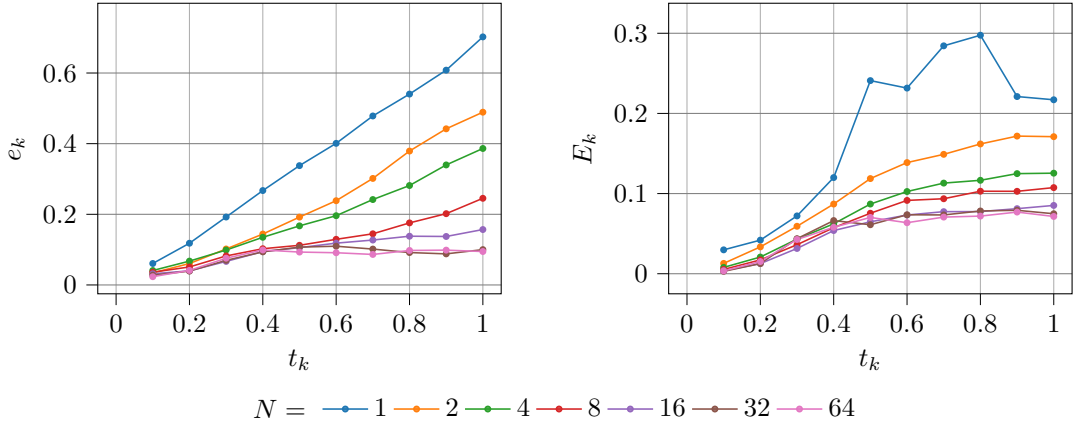

In Figure~\ref{fig: bistable conv plot}, we examine the convergence behavior at the final time $t_K$. Both the final-time error $e_K$ and the accumulated \emph{a posteriori} term $E = \sum_{k=1}^K E_k$ are plotted against a reference slope of $N^{-\frac{1}{2}}$ on a logarithmic scale. The error $e_K$ exhibits an order $\frac{1}{2}$ convergence rate for all but the finest discretization. In contrast, the \emph{a posteriori} term $E$ follows an order $\frac{1}{2}$ decay up to $N = 16$, after which the convergence stagnates. This behavior aligns with the theoretical structure of the error bound in \eqref{main theorem eq: error bound}: once $E$ ceases to decrease, the total error $e_K$ likewise plateaus, reflecting its dependence on $E$ as one of the dominant terms. We emphasize, however, that the coefficients in this example do not satisfy the assumptions of Theorem~\ref{main theorem: error bound}.

The observed plateau of $E$ for large $N$ is likely due to limitations in the practical implementation of the method. Potential causes include insufficient Monte Carlo sampling to estimate the expected value in \eqref{eq: Monte Carlo minimization}, suboptimal network architectures, or challenges in optimization during training. These factors may dominate the discretization error beyond a certain resolution, thus limiting further convergence.

\begin{figure}[h]
\captionsetup{justification=raggedleft}
\begin{subfigure}[b]{0.49\linewidth}
\centering
\input{include/new_figures/BI_Linf_loglog}
\captionsetup{margin=90pt}
\vspace{0pt}
\end{subfigure}
\begin{subfigure}[b]{0.49\linewidth}
\centering
\input{include/new_figures/BI_Post_loglog}
\captionsetup{margin=80pt}
\vspace{0pt}
\end{subfigure}
\begin{tikzpicture}[overlay]
\node at (0,0) {$ $};
\node at (-5.9,0.1) {\ref*{fig: conv legend BI}};
\end{tikzpicture}
\vspace{-10pt}
\captionsetup{justification = justified}
\caption{Bistable process. The figure presents the error and the accumulated residual at the final time for seven different discretizations together with a reference line of order $\frac{1}{2}$. To the left we have $e_K$ and to the right the accumulated residual $E=\sum_{k=1}^K E_k$.}
\vspace{4pt}
\label{fig: bistable conv plot}
\end{figure}
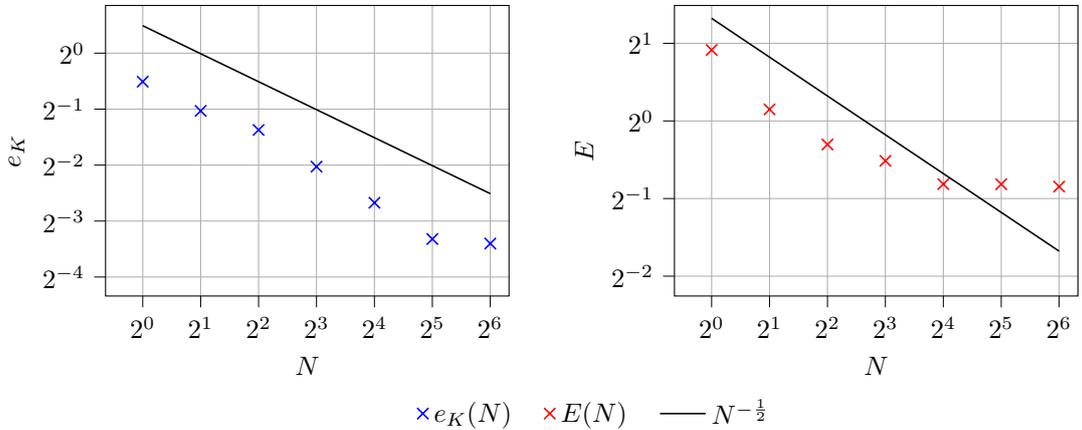

\section{Conclusion and outlook}
\label{section: conclusion}
In this paper, considering the filtering problem in a continuous-discrete setting, we introduce an approximate method, prove an error bound for the method, and empirically confirm the theoretical rate. 
The objective is to develop a filtering approach suitable for high-dimensional and strongly nonlinear settings, where neither Kalman filters nor particle filters are sufficient. This work establishes the theoretical foundation, while a complementary numerical study focusing on high-dimensional problems was conducted concurrently by the first author confirming its numerical robustness \cite{baagmark2025high}.

The numerical results confirm the convergence predicted by Theorem~\ref{main theorem: error bound}. In both the linear and nonlinear examples, the accuracy improves with finer temporal discretization, and for sufficiently fine resolutions, the error, for both left-hand side and right-hand side, remains essentially constant over time. The observed rates match the theoretical prediction until the improvement is limited by factors such as network capacity or training accuracy, rather than by the discretization itself. This indicates that the method achieves the theoretically predicted rates until it is constrained by implementation aspects.

\subsection*{Acknowledgements}
The authors would like to thank Melker Bild for his contributions through initial experiments conducted as part of his master's thesis, which explored related aspects of deep BSDE filtering. The work of {K.B.} and {S.L.} was supported by the Wallenberg AI, Autonomous Systems and Software Program (WASP) funded by the Knut and Alice Wallenberg Foundation. The computations were enabled by resources provided by the National Academic Infrastructure for Supercomputing in Sweden (NAISS) at Chalmers e-Commons partially funded by the Swedish Research Council through grant agreement no. 2022-06725.

\bibliographystyle{abbrv}
\bibliography{bib}

\appendix
\section{Implementation details}
\label{Appendix: implementation details}
In this section, we  present the implementation details. Section~\ref{SM: Networks and training} is on the architecture and parameters and on how the models are trained. Section~\ref{SM: Reference solutions} briefly describes the reference solutions. Section~\ref{SM: Evaluation} describes how the strong error term is evaluated.

\subsection{Networks and training}
\label{SM: Networks and training}
The spaces $\mathbf{NN}^{\Theta,p,k}$, $p=1$ or $p=d$, and $k=0,\dots,K-1$, which define our models, consist of fully connected feed-forward neural networks with three hidden layers, one input layer, and one output layer. The training is organized into $K$ sequential minimization problems, one for each prediction step $k = 0, \dots, K-1$. In each minimization problem, we employ a single network for the $w$ component of the solution and N distinct networks for the v components, one for each time step in the discretization. The architectural differences between the $w$ and $v$ networks are summarized in Table~\ref{tab:network-arch}. To allow parameter transfer between consecutive minimization problems, all models share the same fixed input dimension, independent of the number of observational inputs available at a given step. For time points without observations, the corresponding input components are set to zero. We choose a smaller architecture for the $v$-networks, since each of them models only a short propagation of the process $\mathcal{Y}$, whereas the $w$-network represents the full predictive density after training. The $w$-network uses an exponential output, similar to that of energy-based methods \cite{bagmark_1, Gustafsson_Danelljan}.
\begin{table}[h]
    \centering
    \begin{tabular}{lcc}
        \toprule
        & $w$-network & $v$-networks \\
        \midrule
        Input dimension         & $d + (K-1)\times{d'}$ & $d + (K-1)\times{d'}$ \\
        Hidden layer size       & $128$     & $32$      \\
        Activation function (hidden)       & ReLU     & ReLU      \\
        Activation function (output)       & Exponential     & None     \\
        Output dimension        & $1$       & $d$       \\
        Number of networks      & $1$       & $N$       \\
        \bottomrule
    \end{tabular}
    \caption{Architectural details of the models.}
    \label{tab:network-arch}
\end{table}
Training was carried out with the Adam optimizer with standard parameters, using a fixed learning rate of $10^{-4}$ and a constant batch size of $512$. Each epoch began by generating $200$ batches of simulated trajectories from the process $(S,O,\mathcal{X},W)$, acting as input to the scheme, and evaluating the terminal condition $\overline{g}$, acting as the target label in this learning task. Training was stopped after at most $100$ epochs, or earlier if the averaged loss over the epoch stopped decreasing for $5$ consecutive epochs (early stopping). Thus, we got different $M_{\text{train}}$ values (number of Monte Carlo samples) depending on the number of epochs that were required. We found that the models with the coarser discretizations converged with fewer epochs. Additionally, the cost of training is at least proportional to the number of discretization steps due to the sequential nature of the Euler--Maruyama step when calculating $\mathcal{Y}_N$. The normalization steps performed in both the training and inference time were performed with $J=10^3$ discretization points over a spatial grid on $[-5,5]$. All experiments were run on an NVIDIA A40 GPU (48GB memory).

\subsection{Reference solutions}
\label{SM: Reference solutions}
For the linear example, the Ornstein--Uhlenbeck process, the reference solution was obtained using a discrete-time Kalman filter \cite{Kalman_Bucy}. To reduce discretization error in the prediction step, we used $128$ intermediate prediction steps in each interval $[t_k,t_{k+1}]$; this should be compared to the largest value of the discretization parameter of the model, which was $N = 64$.

For the nonlinear example, the bistable process, we employed a bootstrap particle filter with the same refinement of $128$ intermediate prediction steps for improved accuracy \cite{gordon1993novel}. We used a Gaussian kernel density estimator to recover a probability density \cite{silverman1986density}. A total of $10^5$ particles were used to ensure a high-accuracy estimate of the filtering density and to reduce Monte Carlo variance to a negligible level.

\subsection{Evaluation of errors}
\label{SM: Evaluation}
We approximate the left hand side by considering new samples $(O^{(m)})_{m=1}^{M_{\text{eval}}}$ from \eqref{introduction eq: obs} for $M_{\text{eval}}>0$ and a spatial discretization by equidistant $x_i\in\R$, $i=1,\dots,I$, for $I>0$. We emphasize that we must consider sampled approximations over $\mathbb{O}_k$ since this space is of dimension $10$ for $k=10$ and $d' = 1$. For both examples, it suffices to consider the domain $[-5,5]\subset \R$ for the points $(x_i)_{i=1}^I$.
By these approximations, for $k=1,\dots,K$, we get
\begin{align*}
    \|
    p_k(t_k)
    -
    \widehat{p}_k
    \|_{L^\infty(\mathbb{O}_k;L^\infty(\R^d;\R))}
    \approx
    \underset{
    \substack{
    m=1,\dots,M_{\text{eval}}
    \\
    i=1,\dots,I
    }}
    {\mathrm{max}}
    \Big|
    p_k
    (t_k,x_i,O_{1:k}^{(m)}) 
    -
    \widehat{p}_k
    (x_i,O_{1:k}^{(m)})
    \Big|
    .
\end{align*}
Similarly, we use the samples $(O^{(m)})_{m=1}^{M_{\text{eval}}}$ to evaluate the a posteriori terms $E_k$. This is done through the same Euler--Maruyama steps as in the minimization formulation \eqref{eq: NN minimization}. In the evaluation, we set $I=10^3$ and $M_{\text{eval}} = 10^4$.

\end{document}

%% file: include/new_figures/dbsde_GA.tex
\tikzset{
    mycircle1/.style={
        circle, draw, minimum size=1.0cm, font=\LARGE, align=center, color=blue
    },
    mycircle/.style={
        circle, draw, minimum size=1.0cm, font=\Huge, align=center
    },
    mysquare/.style={
        rectangle, draw, minimum size=1.0cm, font=\large, align=center, color=red
    },
    myarrow/.style={-{Stealth}, line width = 0.3mm}
}

\begin{tikzpicture}[node distance=3.5cm and 2cm]

\def\s{2.2}

%

\node (p0) at (0*\s, 0) {$p_0$};

\node[mysquare] (box0) at (1*\s, 0) {Deep \\ BSDE};
\node (L0) at (2*\s, 1*\s) {$L$};
\node[mycircle1] (mult0) at (2*\s, 0) {$\times$};

\node[mysquare] (box1) at (3*\s, 0) {Deep \\ BSDE};
\node (dots) at (4.25*\s, 0) {$\dots$};

\node (boxk) at (5.5*\s, 0) {};
\node (Lk1) at (5.5*\s, 1*\s) {$L$};
\node[mycircle1] (multk1) at (5.5*\s, 0) {$\times$};
\node[mysquare] (boxk1) at (6.5*\s, 0) {Deep \\ BSDE};
\node (moredots) at (7.75*\s, 0) {};


\coordinate (w0streck) at ($(box0)!0.5!(mult0)$);

\draw[dashed] (-0.25*\s,-0.5) -- (-0.25*\s,-1.9);
\draw[dashed] (1.5*\s,-0.5) -- (1.5*\s,-1.9);
\node (train0) at (0.625*\s,-1.2) {Training 0};
\draw[dashed] (3.5*\s,-0.5) -- (3.5*\s,-1.9);
\node (train1) at (2.5*\s,-1.2) {Training 1};

\draw[dashed] (5*\s,-0.5) -- (5*\s,-1.9);
\draw[dashed] (7*\s,-0.5) -- (7*\s,-1.9);
\node (traink) at (6*\s,-1.2) {Training $k$};

\draw[myarrow] (p0) -- (box0) node[midway, above] {};
\draw[myarrow] (box0) -- (mult0) node[midway, above] {$w_0^*$};
\draw[myarrow] (L0) -- (mult0) node[midway, above] {};
\draw[myarrow] (mult0) -- (box1) node[midway, above] {$\widehat{p}_1$};
\draw[myarrow] (box1) -- ($(dots)+(-1,0)$) node[midway, above] {$w_1^*$};
\draw[myarrow] ($(dots)+(1,0)$) -- (multk1) node[midway, above] {$w_k^*$};
\draw[myarrow] (Lk1) -- (multk1) node[midway, above] {};
\draw[myarrow] (multk1) -- (boxk1) node[midway, above] {$\widehat{p}_{k}$};
\draw[myarrow] (boxk1) -- ($(moredots)+(-1,0)$) node[midway, above] {$w_{k}^*$};





\node[draw, line width= 0.6mm, minimum width=13cm, minimum height=3cm, below=2.0cm of boxk1, align=left, xshift=-5cm, font=\large, color = red, dashed] (bigbox) {};
    \coordinate (tk)  at ($(bigbox.west)+(2.2,0.0)$);
    \coordinate (tkp1) at ($(bigbox.east)+(-2.2,0.0)$);
    \coordinate (right_axis) at ($(bigbox.east)+(-1.0,0.0)$);
    
    \coordinate (rpoints) at ($(tkp1.east)+(-3.2,0.0)$);

    \draw[thick] (tk) ++ (3.4,0) -- (tk);
    \draw[dashed, thick] (tk) ++ (3.4,0) -- (rpoints);
    \draw[->, thick] (tkp1) ++ (-3.4,0) -- (right_axis);
    
    \draw[thick] (tk) ++(0,-0.2) -- ++(0,0.4);
    \draw[thick] (tkp1) ++(0,-0.2) -- ++(0,0.4);
    
    \foreach \i in {1,2,4,5} {
        \draw ($(tk)!{\i/6}!(tkp1)$) ++(0,-0.1) -- ++(0,0.2);
    }
    
    \node[below=0.2cm of tk, font=\Large] {$t_{k,0}$};
    \node[below=0.2cm of tk, xshift=1.5cm, font=\Large] {$t_{k,1}$};
    \node[below=0.2cm of tkp1, xshift=-1.4cm, font=\Large] {$t_{k,N-1}$};
    \node[below=0.2cm of tkp1, font=\Large] {$t_{k,N}$};
    
    \node[above=0.3cm of tk, font=\Large] {$ \mathcal{Y}_{0}^{O_{1:k}} = w_k^*(\mathcal{X}_{0}^k,\, O_{1:k})$};
    \node[above=0.3cm of tkp1, font=\Large] {$\mathcal{Y}_{N}^{O_{1:k}} \approx \widehat{p}_{k}(\mathcal{X}_{N}^k,\, O_{1:k})$};

\draw[dashed, line width= 0.6mm, color=red] (boxk1) -- ($(bigbox) + (5,1.5)$);

\end{tikzpicture}

%% file: include/new_figures/obs_dbsde_schematic.tex
\begin{tikzpicture}[
    font=\sffamily\small,
    timeline/.style={font=\small\bfseries},
    obsbox/.style={draw, thick, fill=orange!30, minimum width=2.2cm, minimum height=0.8cm, rounded corners, drop shadow},
    state/.style={draw, thick, fill=blue!15, minimum width=1.4cm, minimum height=0.8cm, rounded corners, drop shadow},
    brownian/.style={draw, thick, fill=orange!30, minimum width=1.4cm, minimum height=0.8cm, rounded corners, drop shadow},
    nnbox/.style={draw, thick, fill=red!15, minimum width=1.4cm, minimum height=0.8cm, rounded corners, drop shadow},
    wbox/.style={draw, thick, fill=red!15, minimum width=1.4cm, minimum height=0.8cm, rounded corners, drop shadow},
    gradbox/.style={draw, thick, fill=purple!15, minimum width=1.4cm, minimum height=0.8cm, rounded corners, drop shadow},
    output/.style={draw, thick, fill=green!20, minimum width=1.4cm, minimum height=0.8cm, rounded corners, drop shadow},
    func/.style={draw, thick, fill=yellow!20, minimum width=1.4cm, minimum height=0.8cm, rounded corners, drop shadow},
    lossbox/.style={draw, thick, fill=cyan!20, minimum width=1.4cm, minimum height=0.8cm, rounded corners, drop shadow},
    labelstyle/.style={font=\footnotesize\itshape, midway, above},
    arrowdet/.style={-Stealth, thick, blue!60},
    arrowlearn/.style={-Stealth, thick, red!60, dashed},
    arrowdashdet/.style={-Stealth, thick, black, dashed},
    gridnode/.style={minimum width=1.5cm, minimum height=1cm, align=center}
]

\def\dy{2.4}
\def\Ytime{0}
\def\Yrand{\Ytime - 0.5*\dy}
\def\Yx{\Ytime - 1.25*\dy}
\def\Ynn{\Ytime - 2.25*\dy}
\def\Yz{\Ytime - 3.25*\dy}
\def\Yy{\Ytime - 4.25*\dy}
\def\Yfunc{\Ytime - 4.75*\dy}

\def\s{2.8}

\node[anchor=west] at ($(-5.0, \Ytime)$) {\textbf{Time}};
\node[anchor=west] at ($(-5.0, \Yrand)$) {\textbf{Data}};
\node[anchor=west] at ($(-5.0, \Yx)$) {{$\mathcal{X}$}};
\node[anchor=west] at ($(-5.0, \Ynn)$) {\textbf{NNs}};
\node[anchor=west] at ($(-5.0, \Yz)$) {{$\mathcal{Z}$}};
\node[anchor=west] at ($(-5.0, \Yy)$) {$\mathcal{Y}$};

\node[timeline, gridnode] (T0) at (0*\s, \Ytime) {\(t_0\)};
\node[timeline, gridnode] (T1) at (1*\s, \Ytime) {\(t_1\)};
\node[timeline, gridnode] (T2) at (2*\s, \Ytime) {\(t_2\)};
\node[timeline, gridnode] (Tdots) at (3*\s, \Ytime) {\(\cdots\)};
\node[timeline, gridnode] (Tn1) at (4*\s, \Ytime) {\(t_{N-1}\)};
\node[timeline, gridnode] (Tn) at (5*\s, \Ytime) {\(t_N\)};

\node[obsbox, gridnode] (Obs) at ($ (0*\s, \Yrand) + (-2.5, 0) $) {\(O_{1:k}\)};
\node[brownian, gridnode] (X0r) at (0*\s, \Yrand) {\(\mathcal{X}_0\)};
\node[brownian, gridnode] (W1) at (1*\s, \Yrand) {\(\Delta W_{t_1}\)};
\node[brownian, gridnode] (W2) at (2*\s, \Yrand) {\(\Delta W_{t_2}\)};
\node[gridnode] (Wdots) at (3*\s, \Yrand) {\(\cdots\)};
\node[brownian, gridnode] (Wn1) at (4*\s, \Yrand) {\(\Delta W_{t_{N-1}}\)};
\node[brownian, gridnode] (Wn) at (5*\s, \Yrand) {\(\Delta W_{t_N}\)};

\node[state, gridnode] (X0) at (0*\s, \Yx) {\(\mathcal{X}_0\)};
\node[state, gridnode] (X1) at (1*\s, \Yx) {\(\mathcal{X}_1\)};
\node[state, gridnode] (X2) at (2*\s, \Yx) {\(\mathcal{X}_2\)};
\node[gridnode]         (Xdots) at (3*\s, \Yx) {\(\cdots\)};
\node[state, gridnode] (Xn1) at (4*\s, \Yx) {\(\mathcal{X}_{N-1}\)};
\node[state, gridnode] (Xn) at (5*\s, \Yx) {\(\mathcal{X}_N\)};

\node[wbox, gridnode] (Wbox) at ($ (0*\s, \Ynn) + (-2.5, 0) $) {\(w^\theta\)};
\node[nnbox, gridnode] (NN0) at (0*\s, \Ynn) {\(v_0^\theta\)};
\node[nnbox, gridnode] (NN1) at (1*\s, \Ynn) {\(v_1^\theta\)};
\node[nnbox, gridnode] (NN2) at (2*\s, \Ynn) {\(v_2^\theta\)};
\node[gridnode]         (NNdots) at (3*\s, \Ynn) {\(\cdots\)};
\node[nnbox, gridnode] (NNn1) at (4*\s, \Ynn) {\(v_{n-1}^\theta\)};

\node[gradbox, gridnode] (Z0) at (0*\s, \Yz) {\(\mathcal{Z}_0\)};
\node[gradbox, gridnode] (Z1) at (1*\s, \Yz) {\(\mathcal{Z}_1\)};
\node[gradbox, gridnode] (Z2) at (2*\s, \Yz) {\(\mathcal{Z}_2\)};
\node[gridnode]          (Zdots) at (3*\s, \Yz) {\(\cdots\)};
\node[gradbox, gridnode] (Zn1) at (4*\s, \Yz) {\(\mathcal{Z}_{N-1}\)};

\node[output, gridnode] (Y0) at (0*\s, \Yy) {\(\mathcal{Y}_0\)};
\node[output, gridnode] (Y1) at (1*\s, \Yy) {\(\mathcal{Y}_1\)};
\node[output, gridnode] (Y2) at (2*\s, \Yy) {\(\mathcal{Y}_2\)};
\node[gridnode]          (Ydots) at (3*\s, \Yy) {\(\cdots\)};
\node[output, gridnode] (Yn1) at (4*\s, \Yy) {\(\mathcal{Y}_{N-1}\)};
\node[output, gridnode] (Yn) at (5*\s, \Yy) {\(\mathcal{Y}_N\)};

\coordinate (Xm0) at ($(X0)!0.5!(X1)$);
\coordinate (Xm1) at ($(X1)!0.5!(X2)$);
\coordinate (Xm3) at ($(Xdots)!0.5!(Xn1)$);
\coordinate (Xm4) at ($(Xn1)!0.5!(Xn)$);

\coordinate (Ym0) at ($(Y0)!0.5!(Y1)$);
\coordinate (Ym1) at ($(Y1)!0.5!(Y2)$);
\coordinate (Ym2) at ($(Y2)!0.5!(Ydots)$);
\coordinate (Ym3) at ($(Ydots)!0.5!(Yn1)$);
\coordinate (Ym4) at ($(Yn1)!0.5!(Yn)$);

\draw[arrowdashdet] (X0r.south) -- (X0.north);

\draw[arrowdashdet] (X0.south) -- (NN0.north);
\draw[arrowlearn] (NN0.south) -- (Z0.north);
\draw[arrowdet] (Z0.south) -- (Ym0);

\draw[arrowdashdet] (X1.south) -- (NN1.north);
\draw[arrowlearn] (NN1.south) -- (Z1.north);
\draw[arrowdet] (Z1.south) -- (Ym1);

\draw[arrowdashdet] (X2.south) -- (NN2.north);
\draw[arrowlearn] (NN2.south) -- (Z2.north);
\draw[arrowdet] (Z2.south) -- (Ym2);

\draw[arrowdashdet] (Xn1.south) -- (NNn1.north);
\draw[arrowlearn] (NNn1.south) -- (Zn1.north);
\draw[arrowdet] (Zn1.south) -- (Ym4);
\draw[arrowdet] (Zdots.south) -- (Ym3);

\draw[arrowdashdet] (X0.south) -- (Wbox.north);
\draw[arrowlearn] (Wbox.south) -- (Y0.north west);

\draw[arrowdet] (X0) -- (X1);
\draw[arrowdet] (X1) -- (X2);
\draw[arrowdet] (X2) -- (Xdots);
\draw[arrowdet] (Xdots) -- (Xn1);
\draw[arrowdet] (Xn1) -- (Xn);

\draw[arrowdet] (Y0) -- (Y1);
\draw[arrowdet] (Y1) -- (Y2);
\draw[arrowdet] (Y2) -- (Ydots);
\draw[arrowdet] (Ydots) -- (Yn1);
\draw[arrowdet] (Yn1) -- (Yn);

\draw[arrowdet] (W1.west) -| (Xm0);
\draw[arrowdet] (Xm0) -| (Ym0);

\draw[arrowdet] (W2.west) -| (Xm1);
\draw[arrowdet] (Xm1) -| (Ym1);

\draw[arrowdet] (Wn1.west) -| (Xm3);
\draw[arrowdet] (Xm3) -| (Ym3);

\draw[arrowdet] (Wn.west) -| (Xm4);
\draw[arrowdet] (Xm4) -| (Ym4);

\node[func, gridnode] (G) at (6*\s, \Yx) {\(\overline{g}_k(\mathcal{X}_N,O_{1:k})\)};
\node[draw, thick, fill=red!20, minimum width=1.4cm, minimum height=0.8cm, rounded corners, drop shadow, gridnode] (L2) at (6.1*\s, \Yy) {\( \mathbb{E} [ | \overline{g}_k (\mathcal{X}_N, ,O_{1:k}) - \mathcal{Y}_N |^2 ] \)};

\draw[arrowdashdet] (Xn) -- (G);
\draw[arrowdashdet] (Yn) -- (L2);
\draw[arrowdashdet] (G) -- (L2);

\coordinate (OmidW) at ($(Obs)!0.70!(Wbox)$);
\draw[arrowdashdet] (Obs.south) -- (Wbox.north);
\draw[arrowdashdet] (OmidW) -- (NN0.north west);
\draw[arrowdashdet] (OmidW) -- ++(1.2*\s,0) -- (NN1.north west);
\draw[arrowdashdet] (OmidW) -- ++(2.2*\s,0) -- (NN2.north west);
\draw[arrowdashdet] (OmidW) -- ++(3.625*\s,0) ;
\draw[arrowdashdet] (OmidW) + (4.2*\s,0) -- (NNn1.north west);
\draw[arrowdashdet] (4.2*\s,0) ++ (OmidW) -- ++ (2*\s,0) -- (G.south west);

\begin{scope}[shift={(0, \Yfunc)}]
  \draw[arrowdet] (0,0) -- ++(1.2,0) node[right=0.3cm] {Euler--Maruyama step};
  \draw[arrowlearn] (5,0) -- ++(1.2,0) node[right=0.3cm] {Learnable};
  \draw[arrowdashdet] (8.5,0) -- ++(1.2,0) node[right=0.3cm] {Input maps};
\end{scope}

\end{tikzpicture}

%% file: include/new_figures/OU_Linf.tex
\begin{tikzpicture}

\definecolor{crimson2143940}{RGB}{214,39,40}
\definecolor{darkgray176}{RGB}{176,176,176}
\definecolor{darkorange25512714}{RGB}{255,127,14}
\definecolor{forestgreen4416044}{RGB}{44,160,44}
\definecolor{lightgray204}{RGB}{204,204,204}
\definecolor{mediumpurple148103189}{RGB}{148,103,189}
\definecolor{orchid227119194}{RGB}{227,119,194}
\definecolor{sienna1408675}{RGB}{140,86,75}
\definecolor{steelblue31119180}{RGB}{31,119,180}

\begin{axis}[
width=0.45*6.028in,
height=0.45*4.754in,
legend cell align={left},
legend style={
  fill opacity=0.8,
  draw opacity=1,
  text opacity=1,
  at={(0.03,0.97)},
  anchor=north west,
  draw=none
},
legend columns=-1,
tick align=outside,
tick pos=left,
x grid style={darkgray176},
xlabel={$t_k$},
xmajorgrids,
xmin=-0.05, xmax=1.05,
xtick style={color=black},
y grid style={gray},
ylabel={$e_k$},
ymajorgrids,
ymin=-0.0250299139682389, ymax=0.197452039569865,
ytick style={color=black},
legend to name={fig: over time legend OU},
scaled y ticks=false,
y tick label style={/pgf/number format/fixed, /pgf/number format/precision=2}
]
\addplot [semithick, steelblue31119180, mark=*,  mark size=1, mark options={solid}]
table {%
0.100000001490116 0.0275535713881254
0.200000002980232 0.0655442774295806
0.300000011920929 0.105577737092972
0.400000005960464 0.119137413799763
0.5 0.130300521850586
0.600000023841858 0.145892560482025
0.699999988079071 0.160871788859367
0.800000011920929 0.159755676984787
0.899999976158142 0.166380330920219
1 0.178334429860115
};
\addlegendentry{1}
\addplot [semithick, darkorange25512714, mark=*,  mark size=1, mark options={solid}]
table {%
0.100000001490116 0.0195358395576477
0.200000002980232 0.0326341390609741
0.300000011920929 0.0532391034066677
0.400000005960464 0.0655180886387825
0.5 0.0793516337871551
0.600000023841858 0.0907924398779869
0.699999988079071 0.0962423309683799
0.800000011920929 0.0987747982144355
0.899999976158142 0.107053346931934
1 0.104414664208889
};
\addlegendentry{2}
\addplot [semithick, forestgreen4416044, mark=*,  mark size=1, mark options={solid}]
table {%
0.100000001490116 0.0150388544425368
0.200000002980232 0.0218407399952411
0.300000011920929 0.0288917906582355
0.400000005960464 0.0399996675550937
0.5 0.0476452782750129
0.600000023841858 0.0512391105294227
0.699999988079071 0.050485473126173
0.800000011920929 0.0505122952163219
0.899999976158142 0.0591105595231056
1 0.0649075284600257
};
\addlegendentry{4}
\addplot [semithick, crimson2143940, mark=*,  mark size=1, mark options={solid}]
table {%
0.100000001490116 0.0129822846502065
0.200000002980232 0.0147327426820993
0.300000011920929 0.0199437737464904
0.400000005960464 0.0219021216034889
0.5 0.0264957044273614
0.600000023841858 0.0295540448278188
0.699999988079071 0.0285055469721555
0.800000011920929 0.0297771897166967
0.899999976158142 0.03187707811594
1 0.0354757905006408
};
\addlegendentry{8}
\addplot [semithick, mediumpurple148103189, mark=*,  mark size=1, mark options={solid}]
table {%
0.100000001490116 0.0094834277406334
0.200000002980232 0.0106345107778906
0.300000011920929 0.0104010831564664
0.400000005960464 0.0128053948283195
0.5 0.0144955562427639
0.600000023841858 0.013893362134695
0.699999988079071 0.0120802419260144
0.800000011920929 0.0152005432173609
0.899999976158142 0.0158962979912757
1 0.0155473453924059
};
\addlegendentry{16}
\addplot [semithick, sienna1408675, mark=*,  mark size=1, mark options={solid}]
table {%
0.100000001490116 0.0082025481387972
0.200000002980232 0.0105555765330791
0.300000011920929 0.012611218728125
0.400000005960464 0.0132617903873324
0.5 0.0133871333673596
0.600000023841858 0.0140012009069323
0.699999988079071 0.0132855763658881
0.800000011920929 0.0122237876057624
0.899999976158142 0.01164176966995
1 0.0109742013737559
};
\addlegendentry{32}
\addplot [semithick, orchid227119194, mark=*,  mark size=1, mark options={solid}]
table {%
0.100000001490116 0.0071767936460673
0.200000002980232 0.0087373163551092
0.300000011920929 0.0099849104881286
0.400000005960464 0.0102084837853908
0.5 0.0098027102649211
0.600000023841858 0.0102366274222731
0.699999988079071 0.0093899685889482
0.800000011920929 0.0079338504001498
0.899999976158142 0.0077731474302709
1 0.007881230674684
};
\addlegendentry{64}
\end{axis}

\end{tikzpicture}

%% file: include/new_figures/OU_Post.tex
\begin{tikzpicture}

\definecolor{crimson2143940}{RGB}{214,39,40}
\definecolor{darkgray176}{RGB}{176,176,176}
\definecolor{darkorange25512714}{RGB}{255,127,14}
\definecolor{forestgreen4416044}{RGB}{44,160,44}
\definecolor{lightgray204}{RGB}{204,204,204}
\definecolor{mediumpurple148103189}{RGB}{148,103,189}
\definecolor{orchid227119194}{RGB}{227,119,194}
\definecolor{sienna1408675}{RGB}{140,86,75}
\definecolor{steelblue31119180}{RGB}{31,119,180}

\begin{axis}[
width=0.45*6.028in,
height=0.45*4.754in,
legend cell align={left},
legend style={
  fill opacity=0.8,
  draw opacity=1,
  text opacity=1,
  at={(0.03,0.97)},
  anchor=north west,
  draw=none
},
legend columns=-1,
tick align=outside,
tick pos=left,
x grid style={darkgray176},
xlabel={$t_k$},
xmajorgrids,
xmin=-0.05, xmax=1.05,
xtick style={color=black},
y grid style={gray},
ylabel={$E_k$},
ymajorgrids,
ymin=-0.0250299139682389, ymax=0.197452039569865,
ytick style={color=black},
legend to name={NOT USED 1},
scaled y ticks=false,
y tick label style={/pgf/number format/fixed, /pgf/number format/precision=2}
]
\addplot [semithick, steelblue31119180, mark=*,  mark size=1, mark options={solid}]
table {%
0.1 0.019650349393486977
0.2 0.056019704788923264
0.3 0.09588053822517395
0.4 0.1287962645292282
0.5 0.15109458565711975
0.6 0.1627911925315857
0.7 0.17520447075366974
0.8 0.18032851815223694
0.9 0.18573002517223358
1.0 0.18411678075790405
};
\addlegendentry{1}
\addplot [semithick, darkorange25512714, mark=*,  mark size=1, mark options={solid}]
table {%
0.1 0.013847175985574722
0.2 0.038250911980867386
0.3 0.06561241298913956
0.4 0.08882838487625122
0.5 0.10481710731983185
0.6 0.11253099143505096
0.7 0.1228213906288147
0.8 0.12792888283729553
0.9 0.12185613065958023
1.0 0.12760348618030548
};
\addlegendentry{2}
\addplot [semithick, forestgreen4416044, mark=*,  mark size=1, mark options={solid}]
table {%
0.1 0.009811576455831528
0.2 0.028161795809864998
0.3 0.04619405418634415
0.4 0.06157395988702774
0.5 0.07244338095188141
0.6 0.07798375189304352
0.7 0.08373930305242538
0.8 0.08219502866268158
0.9 0.08604670315980911
1.0 0.08675577491521835
};
\addlegendentry{4}
\addplot [semithick, crimson2143940, mark=*,  mark size=1, mark options={solid}]
table {%
0.1 0.007191061973571777
0.2 0.020300444215536118
0.3 0.03386276960372925
0.4 0.04319335147738457
0.5 0.05031861737370491
0.6 0.05648641660809517
0.7 0.05819090083241463
0.8 0.05811295658349991
0.9 0.05844085291028023
1.0 0.060493409633636475
};
\addlegendentry{8}
\addplot [semithick, mediumpurple148103189, mark=*,  mark size=1, mark options={solid}]
table {%
0.1 0.005063983611762524
0.2 0.01395766157656908
0.3 0.023642364889383316
0.4 0.03038715571165085
0.5 0.03527696058154106
0.6 0.03794194757938385
0.7 0.040767841041088104
0.8 0.03980448842048645
0.9 0.041348181664943695
1.0 0.04162648692727089
};
\addlegendentry{16}
\addplot [semithick, sienna1408675, mark=*,  mark size=1, mark options={solid}]
table {%
0.1 0.00356659316457808
0.2 0.010417107492685318
0.3 0.016870060935616493
0.4 0.02190474607050419
0.5 0.0259617418050766
0.6 0.027815047651529312
0.7 0.028479833155870438
0.8 0.028524640947580338
0.9 0.02979849837720394
1.0 0.02983264811336994
};
\addlegendentry{32}
\addplot [semithick, orchid227119194, mark=*,  mark size=1, mark options={solid}]
table {%
0.1 0.0027804772835224867
0.2 0.008264032192528248
0.3 0.012517161667346954
0.4 0.016410578042268753
0.5 0.01816876232624054
0.6 0.019774217158555984
0.7 0.02023204043507576
0.8 0.02031780406832695
0.9 0.020729463547468185
1.0 0.020966149866580963
};
\addlegendentry{64}
\end{axis}

\end{tikzpicture}

%% file: include/new_figures/OU_Linf_loglog.tex
\begin{tikzpicture}

\definecolor{darkgray176}{RGB}{176,176,176}
\definecolor{lightgray204}{RGB}{204,204,204}

\begin{axis}[
width=0.45*6.028in,
height=0.45*4.754in,
legend cell align={left},
legend style={fill opacity=0.8, draw opacity=1, text opacity=1, draw=none},
legend columns=-1,
log basis x={2},
log basis y={2},
tick align=outside,
tick pos=left,
x grid style={darkgray176},
xlabel={$N$},
xmajorgrids,
xmin=0.640896415253715, xmax=80.0546276800871,
xmode=log,
xtick style={color=black},
y grid style={darkgray176},
ylabel={$e_K$},
ymajorgrids,
ymin=0.00494200050408815, ymax=0.40450208646728,
ymode=log,
ytick style={color=black},
legend to name={fig: conv legend OU}
]
\addplot [semithick, blue, mark=x, mark size=3, mark options={solid}, only marks]
table {%
1 0.178334429860115
2 0.104414664208889
4 0.0649075284600257
8 0.0354757905006408
16 0.0155473453924059
32 0.0109742013737559
64 0.007881230674684
};
\addlegendentry{$e_K(N)\quad$}
\addplot [semithick, red, mark=x, mark size=3, mark options={solid}, only marks]
table {%
1 10
};
\addlegendentry{$E(N)\quad$}
\addplot [semithick, black]
table {%
1 0.35666885972023
2 0.252202969346248
4 0.178334429860115
8 0.126101484673124
16 0.0891672149300575
32 0.063050742336562
64 0.0445836074650287
};
\addlegendentry{$N^{-\frac{1}{2}}$}
\end{axis}

\end{tikzpicture}

%% file: include/new_figures/OU_Post_loglog.tex
\begin{tikzpicture}

\definecolor{darkgray176}{RGB}{176,176,176}
\definecolor{lightgray204}{RGB}{204,204,204}

\begin{axis}[
width=0.45*6.028in,
height=0.45*4.754in,
legend cell align={left},
legend style={fill opacity=0.8, draw opacity=1, text opacity=1, draw=none},
legend columns=-1,
log basis x={2},
log basis y={2},
tick align=outside,
tick pos=left,
x grid style={darkgray176},
xlabel={$N$},
xmajorgrids,
xmin=0.640896415253715, xmax=80.0546276800871,
xmode=log,
xtick style={color=black},
y grid style={darkgray176},
ylabel={$E$},
ymajorgrids,
ymin=0.1294200050408815, ymax=2.80450208646728,
ymode=log,
ytick style={color=black},
legend to name={NOT USED 2}
]
\addplot [semithick, red, mark=x, mark size=3, mark options={solid}, only marks]
table {%
1 1.5629
2 1.0061
4 0.6706
8 0.4801
16 0.3379
32 0.2331
64 0.1675
};
\addlegendentry{$E(N)$}
\addplot [semithick, black]
table {%
1 2
64 0.25
};
\addlegendentry{$N^{-\frac{1}{2}}$}
\end{axis}

\end{tikzpicture}

%% file: include/new_figures/BI_Linf.tex
\begin{tikzpicture}

\definecolor{crimson2143940}{RGB}{214,39,40}
\definecolor{darkgray176}{RGB}{176,176,176}
\definecolor{darkorange25512714}{RGB}{255,127,14}
\definecolor{forestgreen4416044}{RGB}{44,160,44}
\definecolor{lightgray204}{RGB}{204,204,204}
\definecolor{mediumpurple148103189}{RGB}{148,103,189}
\definecolor{orchid227119194}{RGB}{227,119,194}
\definecolor{sienna1408675}{RGB}{140,86,75}
\definecolor{steelblue31119180}{RGB}{31,119,180}

\begin{axis}[
width=0.45*6.028in,
height=0.45*4.754in,
legend cell align={left},
legend style={
  fill opacity=0.8,
  draw opacity=1,
  text opacity=1,
  at={(0.03,0.97)},
  anchor=north west,
  draw=none
},
legend columns=-1,
tick align=outside,
tick pos=left,
x grid style={darkgray176},
xlabel={$t_k$},
xmajorgrids,
xmin=-0.05, xmax=1.05,
xtick style={color=black},
y grid style={gray},
ylabel={$e_k$},
ymajorgrids,
ymin=-0.0250299139682389, ymax=0.797452039569865,
ytick style={color=black},
legend to name={fig: over time legend BI},
scaled y ticks=false,
y tick label style={/pgf/number format/fixed, /pgf/number format/precision=2}
]
\addplot [semithick, steelblue31119180, mark=*, mark size=1, mark options={solid}]
table {%
0.100000001490116 0.0607605390250682
0.200000002980232 0.118093185126781
0.300000011920929 0.192290961742401
0.400000005960464 0.267240524291992
0.5 0.33778315782547
0.600000023841858 0.400898516178131
0.699999988079071 0.478385746479034
0.800000011920929 0.540404081344604
0.899999976158142 0.608083069324493
1 0.702251613140106
};
\addlegendentry{1}
\addplot [semithick, darkorange25512714, mark=*, mark size=1, mark options={solid}]
table {%
0.100000001490116 0.0341807603836059
0.200000002980232 0.0606316104531288
0.300000011920929 0.101827837526798
0.400000005960464 0.144106984138489
0.5 0.192098677158356
0.600000023841858 0.238485708832741
0.699999988079071 0.301561802625656
0.800000011920929 0.379052609205246
0.899999976158142 0.44201472401619
1 0.489093244075775
};
\addlegendentry{2}
\addplot [semithick, forestgreen4416044, mark=*, mark size=1, mark options={solid}]
table {%
0.100000001490116 0.0408417731523513
0.200000002980232 0.0675080120563507
0.300000011920929 0.0988634005188942
0.400000005960464 0.134928807616234
0.5 0.167280912399292
0.600000023841858 0.196034178137779
0.699999988079071 0.24178272485733
0.800000011920929 0.28146767616272
0.899999976158142 0.33975139260292
1 0.386334419250488
};
\addlegendentry{4}
\addplot [semithick, crimson2143940, mark=*, mark size=1, mark options={solid}]
table {%
0.100000001490116 0.0360473170876503
0.200000002980232 0.0509448610246181
0.300000011920929 0.0821112021803855
0.400000005960464 0.102546788752079
0.5 0.112244099378586
0.600000023841858 0.129198372364044
0.699999988079071 0.144791916012764
0.800000011920929 0.175666093826294
0.899999976158142 0.20179882645607
1 0.245309382677078
};
\addlegendentry{8}
\addplot [semithick, mediumpurple148103189, mark=*, mark size=1, mark options={solid}]
table {%
0.100000001490116 0.0316134355962276
0.200000002980232 0.0399605855345726
0.300000011920929 0.0666552633047103
0.400000005960464 0.0943477600812912
0.5 0.105615086853504
0.600000023841858 0.118349328637123
0.699999988079071 0.127122581005096
0.800000011920929 0.137982815504074
0.899999976158142 0.13723236322403
1 0.156857848167419
};
\addlegendentry{16}
\addplot [semithick, sienna1408675, mark=*, mark size=1, mark options={solid}]
table {%
0.100000001490116 0.0276506338268518
0.200000002980232 0.0391778722405433
0.300000011920929 0.0697139576077461
0.400000005960464 0.0936990603804588
0.5 0.107090167701244
0.600000023841858 0.109671488404274
0.699999988079071 0.101565144956112
0.800000011920929 0.091715008020401
0.899999976158142 0.0879844203591346
1 0.10005746036768
};
\addlegendentry{32}
\addplot [semithick, orchid227119194, mark=*, mark size=1, mark options={solid}]
table {%
0.100000001490116 0.0233815964311361
0.200000002980232 0.0411717779934406
0.300000011920929 0.0757414028048515
0.400000005960464 0.0990531370043754
0.5 0.0930826142430305
0.600000023841858 0.0911947339773178
0.699999988079071 0.0863903388381004
0.800000011920929 0.0976323187351226
0.899999976158142 0.0985706523060798
1 0.0946451649069786
};
\addlegendentry{64}
\end{axis}

\end{tikzpicture}

%% file: include/new_figures/BI_Post.tex
\begin{tikzpicture}

\definecolor{crimson2143940}{RGB}{214,39,40}
\definecolor{darkgray176}{RGB}{176,176,176}
\definecolor{darkorange25512714}{RGB}{255,127,14}
\definecolor{forestgreen4416044}{RGB}{44,160,44}
\definecolor{lightgray204}{RGB}{204,204,204}
\definecolor{mediumpurple148103189}{RGB}{148,103,189}
\definecolor{orchid227119194}{RGB}{227,119,194}
\definecolor{sienna1408675}{RGB}{140,86,75}
\definecolor{steelblue31119180}{RGB}{31,119,180}

\begin{axis}[
width=0.45*6.028in,
height=0.45*4.754in,
legend cell align={left},
legend style={
  fill opacity=0.8,
  draw opacity=1,
  text opacity=1,
  at={(0.03,0.97)},
  anchor=north west,
  draw=none
},
legend columns=-1,
tick align=outside,
tick pos=left,
x grid style={darkgray176},
xlabel={$t_k$},
xmajorgrids,
xmin=-0.05, xmax=1.05,
xtick style={color=black},
y grid style={gray},
ylabel={$E_k$},
ymajorgrids,
ymin=-0.0250299139682389, ymax=0.337452039569865,
ytick style={color=black},
legend to name={NOT USED 3},
scaled y ticks=false,
y tick label style={/pgf/number format/fixed, /pgf/number format/precision=2}
]
\addplot [semithick, steelblue31119180, mark=*, mark size=1, mark options={solid}]
table {%
0.1 0.029798373579978943
0.2 0.041915372014045715
0.3 0.07213152945041656
0.4 0.11996451765298843
0.5 0.24093681573867798
0.6 0.23160192370414734
0.7 0.2843133807182312
0.8 0.29756999015808105
0.9 0.22104889154434204
1.0 0.2170245349407196
};
\addlegendentry{1}
\addplot [semithick, darkorange25512714, mark=*, mark size=1, mark options={solid}]
table {%
0.1 0.012897484004497528
0.2 0.03349829837679863
0.3 0.0591903030872345
0.4 0.0869314968585968
0.5 0.1187080666422844
0.6 0.1386694759130478
0.7 0.14894753694534302
0.8 0.1618858128786087
0.9 0.17170500755310059
1.0 0.17098137736320496
};
\addlegendentry{2}
\addplot [semithick, forestgreen4416044, mark=*, mark size=1, mark options={solid}]
table {%
0.1 0.00791148841381073
0.2 0.020810524001717567
0.3 0.04315149411559105
0.4 0.062312062829732895
0.5 0.08693436533212662
0.6 0.10256455093622208
0.7 0.11306668072938919
0.8 0.11656250804662704
0.9 0.12488582730293274
1.0 0.12544378638267517
};
\addlegendentry{4}
\addplot [semithick, crimson2143940, mark=*, mark size=1, mark options={solid}]
table {%
0.1 0.005348294973373413
0.2 0.017247773706912994
0.3 0.03621263802051544
0.4 0.0574151873588562
0.5 0.0754670724272728
0.6 0.09146799892187119
0.7 0.09355376660823822
0.8 0.10292912274599075
0.9 0.10281358659267426
1.0 0.1074482798576355
};
\addlegendentry{8}
\addplot [semithick, mediumpurple148103189, mark=*, mark size=1, mark options={solid}]
table {%
0.1 0.003708111820742488
0.2 0.0126029122620821
0.3 0.031521186232566833
0.4 0.053882576525211334
0.5 0.06479468941688538
0.6 0.07322172820568085
0.7 0.07744510471820831
0.8 0.07750208675861359
0.9 0.08117443323135376
1.0 0.08521156013011932
};
\addlegendentry{16}
\addplot [semithick, sienna1408675, mark=*, mark size=1, mark options={solid}]
table {%
0.1 0.0029594635125249624
0.2 0.012678813189268112
0.3 0.04370829090476036
0.4 0.0663168877363205
0.5 0.06112005189061165
0.6 0.07355296611785889
0.7 0.07334164530038834
0.8 0.0782775729894638
0.9 0.07914406061172485
1.0 0.07482649385929108
};
\addlegendentry{32}
\addplot [semithick, orchid227119194, mark=*, mark size=1, mark options={solid}]
table {%
0.1 0.0035723629407584667
0.2 0.015197609551250935
0.3 0.042732786387205124
0.4 0.05797553434967995
0.5 0.0702190175652504
0.6 0.06365595012903214
0.7 0.07066435366868973
0.8 0.07190365344285965
0.9 0.07693210244178772
1.0 0.07143891602754593
};
\addlegendentry{64}
\end{axis}

\end{tikzpicture}

%% file: include/new_figures/BI_Linf_loglog.tex
\begin{tikzpicture}

\definecolor{darkgray176}{RGB}{176,176,176}
\definecolor{lightgray204}{RGB}{204,204,204}

\begin{axis}[
width=0.45*6.028in,
height=0.45*4.754in,
legend cell align={left},
legend style={fill opacity=0.8, draw opacity=1, text opacity=1, draw=none},
legend columns=-1,
log basis x={2},
log basis y={2},
tick align=outside,
tick pos=left,
x grid style={darkgray176},
xlabel={$N$},
xmajorgrids,
xmin=0.640896415253715, xmax=80.0546276800871,
xmode=log,
xtick style={color=black},
y grid style={darkgray176},
ylabel={$e_K$},
ymajorgrids,
ymin=0.0494200050408815, ymax=1.80450208646728,
ymode=log,
ytick style={color=black},
legend to name={fig: conv legend BI}
]
\addplot [semithick, blue, mark=x, mark size=3, mark options={solid}, only marks]
table {%
1 0.702251613140106
2 0.489093244075775
4 0.386334419250488
8 0.245309382677078
16 0.156857848167419
32 0.10005746036768
64 0.0946451649069786
};
\addlegendentry{$e_K(N)\quad$}
\addplot [semithick, red, mark=x, mark size=3, mark options={solid}, only marks]
table {%
1 10
};
\addlegendentry{$E(N)\quad$}
\addplot [semithick, black]
table {%
1 1.40450322628021
2 0.993133755501122
4 0.702251613140106
8 0.496566877750561
16 0.351125806570053
32 0.248283438875281
64 0.175562903285027
};
\addlegendentry{$N^{-\frac{1}{2}}$}
\end{axis}

\end{tikzpicture}

%% file: include/new_figures/BI_Post_loglog.tex
\begin{tikzpicture}

\definecolor{darkgray176}{RGB}{176,176,176}
\definecolor{lightgray204}{RGB}{204,204,204}

\begin{axis}[
width=0.45*6.028in,
height=0.45*4.754in,
legend cell align={left},
legend style={fill opacity=0.8, draw opacity=1, text opacity=1, draw=none},
legend columns=-1,
log basis x={2},
log basis y={2},
tick align=outside,
tick pos=left,
x grid style={darkgray176},
xlabel={$N$},
xmajorgrids,
xmin=0.640896415253715, xmax=80.0546276800871,
xmode=log,
xtick style={color=black},
y grid style={darkgray176},
ylabel={$E$},
ymajorgrids,
ymin=0.2094200050408815, ymax=2.80450208646728,
ymode=log,
ytick style={color=black},
legend to name={NOT USED 4}
]
\addplot [semithick, red, mark=x, mark size=3, mark options={solid}, only marks]
table {%
1 1.8853
2 1.1087
4 0.8103
8 0.7004
16 0.5687
32 0.5683
64 0.5563
};
\addlegendentry{$E(N)$}
\addplot [semithick, black]
table {%
1 2.5
64 0.3125
};
\addlegendentry{$N^{-\frac{1}{2}}$}
\end{axis}

\end{tikzpicture}

%% file: main.bbl
\begin{thebibliography}{10}

\bibitem{andersson2025deep}
K.~Andersson, A.~Andersson, and C.~W. Oosterlee.
\newblock The deep multi-{FBSDE} method: a robust deep learning method for coupled {FBSDEs}.
\newblock {\em arXiv:2503.13193}, 2025.

\bibitem{bagmark_1}
K.~B{\aa}gmark, A.~Andersson, and S.~Larsson.
\newblock An energy-based deep splitting method for the nonlinear filtering problem.
\newblock {\em Partial Differ. Equ. Appl.}, 4(2):14, 2023.

\bibitem{baagmark2024convergent}
K.~B{\aa}gmark, A.~Andersson, S.~Larsson, and F.~Rydin.
\newblock A convergent scheme for the {B}ayesian filtering problem based on the {F}okker--{P}lanck equation and deep splitting.
\newblock {\em arXiv:2409.14585}, 2024.

\bibitem{baagmark2025high}
K.~B{\aa}gmark and F.~Rydin.
\newblock High-dimensional {B}ayesian filtering through deep density approximation.
\newblock {\em arXiv:2511.07261}, 2025.

\bibitem{barshalom2001estimation}
Y.~Bar-Shalom, X.~R. Li, and T.~Kirubarajan.
\newblock {\em Estimation with Applications to Tracking and Navigation}.
\newblock John Wiley \& Sons, 2001.

\bibitem{Arnulf}
C.~Beck, S.~Becker, P.~Cheridito, A.~Jentzen, and A.~Neufeld.
\newblock Deep learning based numerical approximation algorithms for stochastic partial differential equations and high-dimensional nonlinear filtering problems.
\newblock {\em arXiv:2012.01194}, 2020.

\bibitem{Arnulf_PDE}
C.~Beck, S.~Becker, P.~Cheridito, A.~Jentzen, and A.~Neufeld.
\newblock Deep splitting method for parabolic {PDE}s.
\newblock {\em SIAM J. Sci. Comput.}, 43:A3135--A3154, 2021.

\bibitem{blackman1999design}
S.~S. Blackman and R.~Popoli.
\newblock {\em Design and Analysis of Modern Tracking Systems}.
\newblock Artech House Publishers, 1999.

\bibitem{challa2000nonlinear}
S.~Challa and Y.~Bar-Shalom.
\newblock Nonlinear filter design using {F}okker-{P}lanck-{K}olmogorov probability density evolutions.
\newblock {\em IEEE Trans. Aerosp. Electron. Syst.}, 36:309--315, 2000.

\bibitem{chopin2023computational}
N.~Chopin, A.~Fulop, J.~Heng, and A.~H. Thiery.
\newblock Computational {D}oob h-transforms for online filtering of discretely observed diffusions.
\newblock In {\em Int. Conf. Mach. Learn.}, pages 5904--5923. PMLR, 2023.

\bibitem{corenflos2024particlemala}
A.~Corenflos and A.~Finke.
\newblock Particle-{MALA} and particle-m{GRAD}: Gradient-based {MCMC} methods for high-dimensional state-space models.
\newblock {\em arXiv:2401.14868}, 2024.

\bibitem{corenflos2024conditioning}
A.~Corenflos, Z.~Zhao, T.~B. Sch{\"o}n, S.~S{\"a}rkk{\"a}, and J.~Sj{\"o}lund.
\newblock Conditioning diffusion models by explicit forward-backward bridging.
\newblock In {\em Int. Conf. Artif. Intell. Stat.}, pages 3709--3717. PMLR, 2025.

\bibitem{Crisan_Lobbe}
D.~Crisan, A.~Lobbe, and S.~Ortiz-Latorre.
\newblock An application of the splitting-up method for the computation of a neural network representation for the solution for the filtering equations.
\newblock {\em Stoch. Partial Differ. Equ.: Anal. Comput.}, 10:1050--1081, 2022.

\bibitem{cui2005comparison}
N.~Cui, L.~Hong, and J.~R. Layne.
\newblock A comparison of nonlinear filtering approaches with an application to ground target tracking.
\newblock {\em Signal Processing}, 85:1469--1492, 2005.

\bibitem{demissie2016nonlinear}
B.~Demissie, M.~A. Khan, and F.~Govaers.
\newblock {Nonlinear filter design using Fokker-Planck propagator in Kronecker tensor format}.
\newblock In {\em 2016 19th International Conference on Information Fusion (FUSION)}, pages 1--8. IEEE, 2016.

\bibitem{E_2017}
W.~E, J.~Han, and A.~Jentzen.
\newblock Deep learning-based numerical methods for high-dimensional parabolic partial differential equations and backward stochastic differential equations.
\newblock {\em Commun. Math. Stat}, 5:349--380, Nov. 2017.

\bibitem{e2017deep}
W.~E and B.~Yu.
\newblock The deep {R}itz method: A deep learning-based numerical algorithm for solving variational problems.
\newblock {\em Commun. Math. Stat}, 1:1--12, 2018.

\bibitem{BSDE_in_finance}
N.~El~Karoui, S.~Peng, and M.~C. Quenez.
\newblock Backward stochastic differential equations in finance.
\newblock {\em Math. Finance}, 7(1):1--71, 1997.

\bibitem{evensen2009data}
G.~Evensen.
\newblock {\em Data Assimilation: The Ensemble Kalman Filter}.
\newblock Springer, 2009.

\bibitem{finke2023conditional}
A.~Finke and A.~H. Thiery.
\newblock {Conditional sequential Monte Carlo in high dimensions}.
\newblock {\em Ann. Statist.}, 51:437--463, 2023.

\bibitem{goodman1997mathematics}
I.~R. Goodman, R.~P.~S. Mahler, and H.~T. Nguyen.
\newblock {\em Mathematics of Data Fusion}.
\newblock Kluwer Academic Publishers Group, Dordrecht, 1997.

\bibitem{gordon1993novel}
N.~J. Gordon, D.~J. Salmond, and A.~F.~M. Smith.
\newblock Novel approach to nonlinear/non-{G}aussian {B}ayesian state estimation.
\newblock {\em IEE Proceedings F (Radar and Signal Processing)}, 140(2):107--113, 1993.

\bibitem{Gustafsson_Danelljan}
F.~K. Gustafsson, M.~Danelljan, G.~Bhat, and T.~B. Sch{\"o}n.
\newblock Energy-based models for deep probabilistic regression.
\newblock In {\em European Conference on Computer Vision}, pages 325--343. Springer, 2020.

\bibitem{han2024deep}
J.~Han, W.~Hu, J.~Long, and Y.~Zhao.
\newblock Deep {P}icard iteration for high-dimensional nonlinear {PDE}s.
\newblock {\em SIAM J. Sci. Comput.}, 48(1):C1--C24, 2026.

\bibitem{han_convergence}
J.~Han and J.~Long.
\newblock Convergence of the deep {BSDE} method for coupled {FBSDE}s.
\newblock {\em Probab. Uncertain. Quant. Risk}, 5:Paper No. 5, 33, 2020.

\bibitem{isard1998condensation}
M.~Isard and A.~Blake.
\newblock Condensation—conditional density propagation for visual tracking.
\newblock {\em Int. J. Comput. Vis.}, 29(1):5--28, 1998.

\bibitem{jasra2005population}
A.~Jasra, D.~A. Stephens, and C.~C. Holmes.
\newblock Population-based reversible jump {M}arkov chain {M}onte {C}arlo.
\newblock {\em Biometrika}, 92(4):803--820, 2005.

\bibitem{johannes2009mcmc}
M.~S. Johannes and N.~G. Polson.
\newblock {MCMC Methods for Continuous-Time Financial Econometrics}.
\newblock In {\em Handbook of Financial Econometrics}, pages 1--72. Elsevier, 2009.

\bibitem{Kalman_Bucy}
R.~E. Kalman and R.~S. Bucy.
\newblock New results in linear filtering and prediction theory.
\newblock {\em J. Basic Eng.}, 83:95--108, 1961.

\bibitem{kapllani2025backward}
L.~Kapllani and L.~Teng.
\newblock A backward differential deep learning-based algorithm for solving high-dimensional nonlinear backward stochastic differential equations.
\newblock {\em IMA J. Numer. Anal.}, 2025.

\bibitem{kording2004bayesian}
K.~P. K{\"o}rding and D.~M. Wolpert.
\newblock Bayesian integration in sensorimotor learning.
\newblock {\em Nature}, 427(6971):244--247, 2004.

\bibitem{krishnapriyan2021characterizing}
A.~Krishnapriyan, A.~Gholami, S.~Zhe, R.~Kirby, and M.~W. Mahoney.
\newblock Characterizing possible failure modes in physics-informed neural networks.
\newblock {\em Adv. Neural Inf. Process. Syst.}, 34:26548--26560, 2021.

\bibitem{Lu_2021}
L.~Lu, P.~Jin, G.~Pang, Z.~Zhang, and G.~E. Karniadakis.
\newblock {Learning nonlinear operators via DeepONet based on the universal approximation theorem of operators}.
\newblock {\em Nat. Mach. Intell.}, 3:218--229, 2021.

\bibitem{luo2025physics}
K.~Luo, J.~Zhao, Y.~Wang, J.~Li, J.~Wen, J.~Liang, H.~Soekmadji, and S.~Liao.
\newblock Physics-informed neural networks for {PDE} problems: a comprehensive review.
\newblock {\em Artif. Intell. Rev.}, 58(10):1--43, 2025.

\bibitem{maybeck1979stochastic}
P.~S. Maybeck.
\newblock {\em Stochastic Models, Estimation, and Control, Volume 1}.
\newblock Academic Press, 1979.

\bibitem{naesseth2019high}
C.~A. Naesseth, F.~Lindsten, and T.~B. Sch{\"o}n.
\newblock {High-dimensional filtering using nested sequential Monte Carlo}.
\newblock {\em IEEE Trans. Signal Process.}, 67:4177--4188, 2019.

\bibitem{Oksendal}
B.~{\O}ksendal.
\newblock {\em Stochastic Differential Equations: An Introduction with Applications}.
\newblock Springer Science \& Business Media, 2003.

\bibitem{raissi2019physics}
M.~Raissi, P.~Perdikaris, and G.~E. Karniadakis.
\newblock Physics-informed neural networks: {A} deep learning framework for solving forward and inverse problems involving nonlinear partial differential equations.
\newblock {\em J.~Comput. Phys.}, 378:686--707, 2019.

\bibitem{schauer2017guided}
M.~Schauer, F.~van~der Meulen, and H.~van Zanten.
\newblock Guided proposals for simulating multi-dimensional diffusion bridges.
\newblock {\em Bernoulli}, 23(4A):2917--2950, 2017.

\bibitem{silverman1986density}
B.~Silverman.
\newblock {\em Density Estimation for Statistics and Data Analysis}.
\newblock Chapman \& Hall/CRC, 1986.

\bibitem{snyder2011particle}
C.~Snyder.
\newblock Particle filters, the “optimal” proposal and high-dimensional systems.
\newblock In {\em Proceedings of the ECMWF Seminar on Data Assimilation for atmosphere and ocean}, pages 1--10, 2011.

\bibitem{snyder2015performance}
C.~Snyder, T.~Bengtsson, and M.~Morzfeld.
\newblock Performance bounds for particle filters using the optimal proposal.
\newblock {\em Mon. Weather Rev.}, 143:4750--4761, 2015.

\bibitem{thrun2005probabilistic}
S.~Thrun, W.~Burgard, and D.~Fox.
\newblock {\em Probabilistic Robotics}.
\newblock MIT Press, 2005.

\bibitem{zhao2024conditional}
Z.~Zhao, Z.~Luo, J.~Sj\"{o}lund, and T.~B. Sch\"{o}n.
\newblock Conditional sampling within generative diffusion models.
\newblock {\em Philos. Trans. R. Soc. A}, 383(2299):20240329, 2025.

\end{thebibliography}
